\documentclass[11pt, letterpaper, reqno]{amsart}

\usepackage{color}
\usepackage{amsfonts}
\usepackage{amssymb}
\usepackage{amsmath}
\usepackage{hyperref}

\usepackage{graphicx}
\usepackage{enumitem}
\usepackage{subcaption}

\usepackage[margin=1in]{geometry}

\newcommand{\df}{\buildrel\rm{def}\over=}

\newcommand{\pd}{\partial}

\def\supp{\operatorname{supp}}

\newcommand{\al}{\alpha}
\newcommand{\eps}{\varepsilon}
\newcommand{\la}{\lambda}
\newcommand{\vf}{\varphi}
\newcommand{\om}{\omega}

\newcommand{\ep}{\epsilon}

\newcommand{\cD}{\mathcal D}
\newcommand{\cE}{\mathcal E}

\newcommand{\cR}{\mathcal R}
\newcommand{\cS}{\mathcal S}

\newcommand{\bI}{\mathbb I}

\newcommand{\bR}{\mathbb R}

\newcommand{\bV}{\mathbb V}
\newcommand{\bZ}{\mathbb Z}

\newtheorem{theorem}{Theorem}[section]

\newtheorem{lemma}[theorem]{Lemma}

\theoremstyle{definition}
\newtheorem{remark}[theorem]{Remark}
\newtheorem{defin}[theorem]{Definition}

\numberwithin{equation}{section}

\newcounter{vremennyj}

\begin{document}

\title[Bi-parameter embedding without capacity]%
{
{Bi-parameter embedding and measures with restricted energy conditions}
}
\author[N.~Arcozzi]{Nicola Arcozzi}
\address{Universit\`{a} di Bologna, Department of Mathematics, Piazza di Porta S. Donato, 40126 Bologna (BO)}
\email{nicola.arcozzi@unibo.it}
\thanks{Theorem 3.1 was obtained in the frameworks of the project 17-11-01064 by the Russian Science Foundation}
\thanks{NA is partially supported by the grants INDAM-GNAMPA 2017 "Operatori e disuguaglianze integrali in spazi con simmetrie" and PRIN 2018 "Variet\`{a} reali e complesse: geometria, topologia e analisi armonica"}
\author[I.~ Holmes]{Irina Holmes}
\thanks{IH is partially supported by the NSF an NSF Postdoc under Award No.1606270}
\address{Department of Mathematics, Michigan Sate University, East Lansing, MI. 48823}
\author[P. Mozolyako]{Pavel Mozolyako}
\thanks{PM is supported by the Russian Science Foundation grant 17-11-01064}
\address{Universit\`{a} di Bologna, Department of Mathematics, Piazza di Porta S. Donato, 40126 Bologna (BO)}
\email{pavel.mozolyako@unibo.it}
\author[A.~Volberg]{Alexander Volberg}
\thanks{AV is partially supported by the NSF grant DMS-160065}
\address{Department of Mathematics, Michigan Sate University, East Lansing, MI. 48823}
\email{volberg@math.msu.edu \textrm{(A.\ Volberg)}}
\makeatletter
\@namedef{subjclassname@2010}{
  \textup{2010} Mathematics Subject Classification}
\makeatother
%
%

\begin{abstract}
Nicola Arcozzi, Pavel Mozolyako, Karl-Mikael Perfekt, and Giulia Sarfatti recently gave  the proof of a bi-parameter Carleson embedding theorem. Their proof uses heavily the notion of capacity on the bi-tree. In this note we give another proof of a bi-parameter Carleson embedding theorem that avoids the use of bi-tree capacity. Unlike the proof on a simple tree in a previous paper of the authors, which used the Bellman function technique, the proof here is based on some rather subtle comparisons of energies of measures on the bi-tree. 
\end{abstract}
\maketitle

\section{Introduction and Notations}
\label{intro}

Let $T$ denote a finite dyadic tree (of depth $N$). By identifying the root of $T$ with $I_0 = [0, 1]$ and each subsequent node in $T$ with the corresponding dyadic subinterval of $I_0$, we can think of its boundary $\pd T$ as simply $\cD_N$, i.e. the dyadic subintervals of $I_0$ of size $2^{-N}$. 

Consider now $T^2 = T \times T$, a bi-tree. We identify the root of $T^2$ with $Q_0=[0,1]^2$, and then each node $\al \in T^2$ of the bi-tree is identified with a corresponding dyadic rectangle $R_{\al} \subset Q_0$ in the obvious way. If $\al$ and $\beta$ are nodes of the bi-tree $T^2$, we say that $\al \leq \beta$ if and only if their corresponding dyadic rectangles satisfy $R_{\al} \subset R_{\beta}$. 

The boundary $\pd T^2 = (\pd T)^2$ will consist of $\cD_N\times \cD_N$, the dyadic sub-squares of $Q_0$ of side-length $2^{-N}$. 
We will usually denote these boundary nodes $\pd T^2$ by the letter $\omega$. The small  squares of size $2^{-n}\times 2^{-N}$ making up the boundary will be denoted $R_\om$. In fact, as the nodes of the bi-tree $T^2$ and dyadic rectangles are in one-to-one correspondence, we will feel free in what follows to sometimes replace the symbol $R_\om$ by  just $\om$,  and $R_\al$ by just $\al$. This should not lead to a confusion, and sometimes it is nice to distinguish between the two objects.

If $E$ is a subset of $\pd T^2$ (or $\pd T$), then we define
$U_E$ to be the union of corresponding squares (intervals for $T$):
		$$U_E := \bigcup_{\omega\in E}R_{\omega}, \:\: \forall E \subset \pd T^2,$$
and $\cR_E$ to be the collection of all dyadic rectangles inside $U_E$ (this is a collection of dyadic intervals if we mean $T$ instead of $T^2$):
	$$\cR_E := \{R  : R \subset U_E\}.$$
	

We consider measures $\mu$ on $\pd T^2$ (or on $\pd T$) that have \textit{constant density} on each small square $R_\omega \in \pd T^2$ (or small interval of $\pd T$). Then if $R\in \cR_E$, obviously
	$$ \mu(R) = \sum_{\om \in E, \,\om \subset R}\mu(R_\om).$$
We can also interpret $\mu$ in terms of the nodes of the bi-tree. For this, recall the Hardy operator 
$\bI: \ell^2(T^2) \rightarrow \ell(T^2)$
on a bi-tree: 
for any $\vf: T^2\to \bR$ let
	$$ \bI \vf (\al)= \sum_{\beta\ge \al} \vf(\beta).$$
Correspondingly it is defined on $T$, but then it is called $I$. Its dual $\bI^*$ is given by the formula
	$$ \bI^* \psi (\al)= \sum_{\beta\le \al} \psi(\beta).$$
Then, of course,
	$$ \mu(R_{\al}) = (\bI^*\mu) (\al).$$

\begin{remark}
The equality above needs perhaps a small clarification, specifically in the last step below:
	$$(\bI^*\mu) (\al) = \sum_{\beta \leq \al} \mu(\beta) = \sum_{\omega\in \pd T^2; \: \omega \leq \alpha} \mu(\omega).$$
The vertices (nodes) of the bi-tree and the dyadic rectangles are the same things (the same  can be said about the nodes of the tree $T$ and the dyadic intervals).
However, notice that given $\al\in T^2\setminus \pd T^2$ (or $\al\in T\setminus \pd T$) we distinguish between $\mu(\al)$ and $\mu(R_\al)$. In fact, $\mu(\al) = 0$ for all $\al\in T^2\setminus\pd T^2$ (or $\al\in T\setminus \pd T$ if we consider just a tree and not a bi-tree). This is because we assume from the start that the measure lies on the boundary of the tree. On the other hand, 
$$
\mu(R_\al)= \sum_{\om\in \pd T^2,\, \om \subset R_\al} \mu(R_\om) = \sum_{\om \in \pd T^2,\, \om \le\al} \mu(\om)\,.
$$
At the same time, if $\om \in \pd T^2$ (or $\om \in \pd T$), then $\mu(\om)=\mu(R_\om)$.
\end{remark}

\begin{remark}
As we already mentioned, we assume from the start that the measure lies on the boundary of the tree. The results of this paper extend to the case when $\mu$ is given on the whole $T^2$, this is done on our subsequent article.
\end{remark}

\begin{defin}
\label{def:Cm}
We say that a measure $\mu$ on $\pd T^2$ (or $\pd T$) is a \textbf{$C$-Carleson measure} if
for any subset $E\subset \pd T^2$ we have
	$$ \sum_{R \in \cR_E} \mu(R)^2 \le C\mu(E).$$
Of course we can give the analogous definitions for a simple tree $T$.
\end{defin}

This is just the condition \eqref{bIstar} below, when it is tested on characteristic functions. Sometimes  it is called ``the dual testing condition'' in the literature.


\begin{defin}
\label{def:CmH}
We say that a measure $\mu$ on $\pd T^2$ (or $\pd T$) is a \textbf{hereditary Carleson measure} if 
there exists a constant $C$ such that $\mu|E$ is $C$-Carleson for any subset $E \subset \pd T^2$ (or $\pd T$).
Here $\mu|E$ denotes the restriction of $\mu$ to $E$:
	\begin{equation*}
	(\mu|E)(\om) := \left\{ \begin{array}{ll}
		\mu(\om), & \text{ if } \om\in E,\\
		0, & \text{ if } \om\not\in E.
		\end{array}\right.
	\end{equation*}
So, in terms of rectangles,
	\begin{equation}\label{E:rest}
	(\mu|E)(R_\al) = \mu(R_\al\cap U_E).
	\end{equation}
The hereditary Carleson condition can then be restated as:
	\begin{equation}\label{E:her2}
	\sum_{R\in\cR_F} \mu(R\cap U_E)^2 \leq C \mu(E\cap F), \:\: \forall E, F \subset \pd T^2.
	\end{equation}
\end{defin}

It is proved in \cite{AMPS} that to be a Carleson measure on $\pd T^2$ is the same as to be a capacitary measure.
Capacitary property is hereditary, and so any Carleson measure on $\pd T^2$ (or $\pd T$) is hereditary Carleson.
However, the main goal of this note is to avoid  the use of capacity, and to prove directly the following result.

\begin{theorem}
\label{main}
Let $\mu$ be a measure on $\pd T^2$.
Then the following are equivalent: 
	\begin{enumerate}
	\item $\mu$ is Carleson;
	\item $\mu$ is hereditary Carleson;
	\item $\mu$ is an embedding measure for the Hardy operator, in the sense that 
		\begin{equation} \label{bI}
		\sum_{\om\in \pd T^2} |\bI \vf(\om)|^2 \mu(\om) \le C_1 \|\vf\|^2_{\ell^2(T^2)};
		\end{equation}
	\item $\mu$ satisfies the {\bf second embedding}:
		\begin{equation} \label{bIstar}
		\sum_{\al\in T^2} |\bI^*( \psi\mu)(\alpha) |^2\le C_1 \sum_{\om \in \pd T^2} |\psi(\om)|^2 \mu (\om).
		\end{equation}
	\end{enumerate}
\end{theorem}

There are some easy implications, like \textit{(2)} obviously implies \textit{(1)}. The fact that \textit{(3)} is equivalent to
\textit{(4)} is just duality: note that \textit{(3)} is the same as the boundedness of the operator $\bI : \ell^2(T^2) \rightarrow \ell(T^2;\: \mu)$, where the inner product in the latter is given by
	$$(\vf, \psi)_{\ell^2(T^2;\: \mu)} := \sum_{\al\in T^2} \vf(\al)\psi(\al)\mu(\al) 
		= \sum_{\om\in\pd T^2} \vf(\om) \psi(\om) \mu(\om).$$
Since
	$$(\bI\vf, \psi)_{\ell^2(T^2;\:\mu)} = \sum_{\al\in T^2} \vf(\al) \bI^*(\psi\mu)(\al)
		= (\vf, \bI^*(\psi\mu))_{\ell^2(T^2)},$$
the adjoint of $\bI : \ell^2(T^2) \rightarrow \ell(T^2;\: \mu)$ is then
	$\bI^*(\cdot\mu): \ell^2(T;\:\mu) \rightarrow \ell^2(T^2)$, 
and \textit{(4)} is exactly boundedness of this operator.

Also the claim that \textit{(4)} implies \textit{(1)} is easy: let some $E\subset \pd T^2$ and choose in \textit{(4)} the function
	$$
	\psi(\al) := \left\{\begin{array}{ll}
		1, & \text{ if } \al\in E\\
		0, & \text{ otherwise.}
	\end{array}\right.
	$$
Then $\bI^*(\psi\mu)(\al) = \mu(R_\al\cap E)$, and  \eqref{bIstar} becomes
	\begin{equation}
	\label{rec1}
	\sum_{\al\in T^2 }\mu(R_\al \cap U_E)^2 \le C_1 \mu(E).
	\end{equation}
Obviously the left hand side is greater than $\sum_{R\in\cR_E}\mu(R\cap U_E)^2 = \sum_{R\in\cR_E}\mu(R)^2$, and we then have the $C_1$-Carleson property. We briefly remark here that the relationship in \eqref{rec1} describes exactly the notion of restricted energy condition, which we will encounter shortly.

The implication \textit{(3)} $\Rightarrow$ \textit{(2)} now is also easy: if the measure $\mu$ satisfies \eqref{bI}, then obviously any measure smaller than $\mu$ also must satisfy \eqref{bI}. So, if $\mu$ satisfies \textit{(3)} then for every $E\subset \pd T^2$, the measure
$\mu|E$ also satisfies \textit{(3)} -- therefore also \textit{(4)}, which we showed implies \textit{(1)}. Then $\mu|E$ is $C_1$-Carleson for all $E\subset\pd T^2$, proving that $\mu$ is hereditary Carleson. Alternatively, one can take in \textit{(4)} the function $\psi(\al)$ which is $1$ when $\al\in E\cap F$ and $0$ elsewhere. As seen in \eqref{rec1}, this will give us
	$$\sum_{\al\in T^2}\mu(R_\al\cap U_{E\cap F})^2 \leq C_1 \mu(E\cap F),$$ 
which then easily implies the hereditary Carleson condition in \eqref{E:her2}.

The difficult implications are \textit{(1)} $\Rightarrow$ \textit{(2)} and \textit{(2)} $\Rightarrow$ \textit{(3)}.
To illustrate that \textit{(1)} $\Rightarrow$ \textit{(2)} is highly non-trivial, let us consider the  simple case of $T$ (much simpler than the bi-tree $T^2$ case). The Carleson property \textit{(1)} is the same as
\begin{equation}
\label{TCarl}
\forall J\in \cD,\,\, \sum_{I\in \cD(J)} \mu(I)^2 \le C\mu(J).
\end{equation}
Let us choose a dyadic interval $K$ and let $\nu=\mu|K$. If we believe that \eqref{TCarl} is hereditary (may be with another constant) then, in particular,
 $$
 \sum_{L\in \cD(I_0),\, K\subset L} \nu(L)^2 \le C' \nu(I_0),
 $$
 but clearly $\nu(L)= \mu(K), \nu(I_0)=\mu(K)$ and we obtain   that $g(K) \mu(K)^2 \le C' \mu(K)$. Here $g(K)=\log1/|K|$, that is the number of the dyadic generation of $K$. Thus, we get
 \begin{equation}
 \label{muK}
 \mu \in \eqref{TCarl} \Rightarrow \mu(K) \le \frac{C'}{g(K)} =\frac{C'}{\log1/|K|}.
 \end{equation}
 
 One can indeed  deduce \eqref{muK} from the Carleson property
 \eqref{TCarl} directly, but it requires some real work, see \cite{HPV}. Moreover, in \cite{HPV} we deduced the box capacitary condition
 $$
  \mu(K\times J) \le \frac{C'}{\log1/|K|\log1/|J|}.
 $$
 from the box condition on bi-tree:
 \begin{equation}
\label{bocCarl}
\forall R_0\in \cD\times \cD,\,\, \sum_{R\in \cD(R_0)} \mu(R)^2 \le C\mu(R_0).
\end{equation}


\subsection{Restricted Energy Condition} 

At this point we are in situation (A) of Figure \ref{fig:Summ} below, namely we are left with the difficult implications 
\textit{(1)} $\Rightarrow$ \textit{(2)} and \textit{(2)} $\Rightarrow$ \textit{(3)}.
We will prove these by appealing to a fourth concept, that of restricted energy condition, 
which we introduce next.

\begin{figure}[h!]
  \centering
  \begin{subfigure}[b]{0.45\linewidth}
    \includegraphics[width=\linewidth]{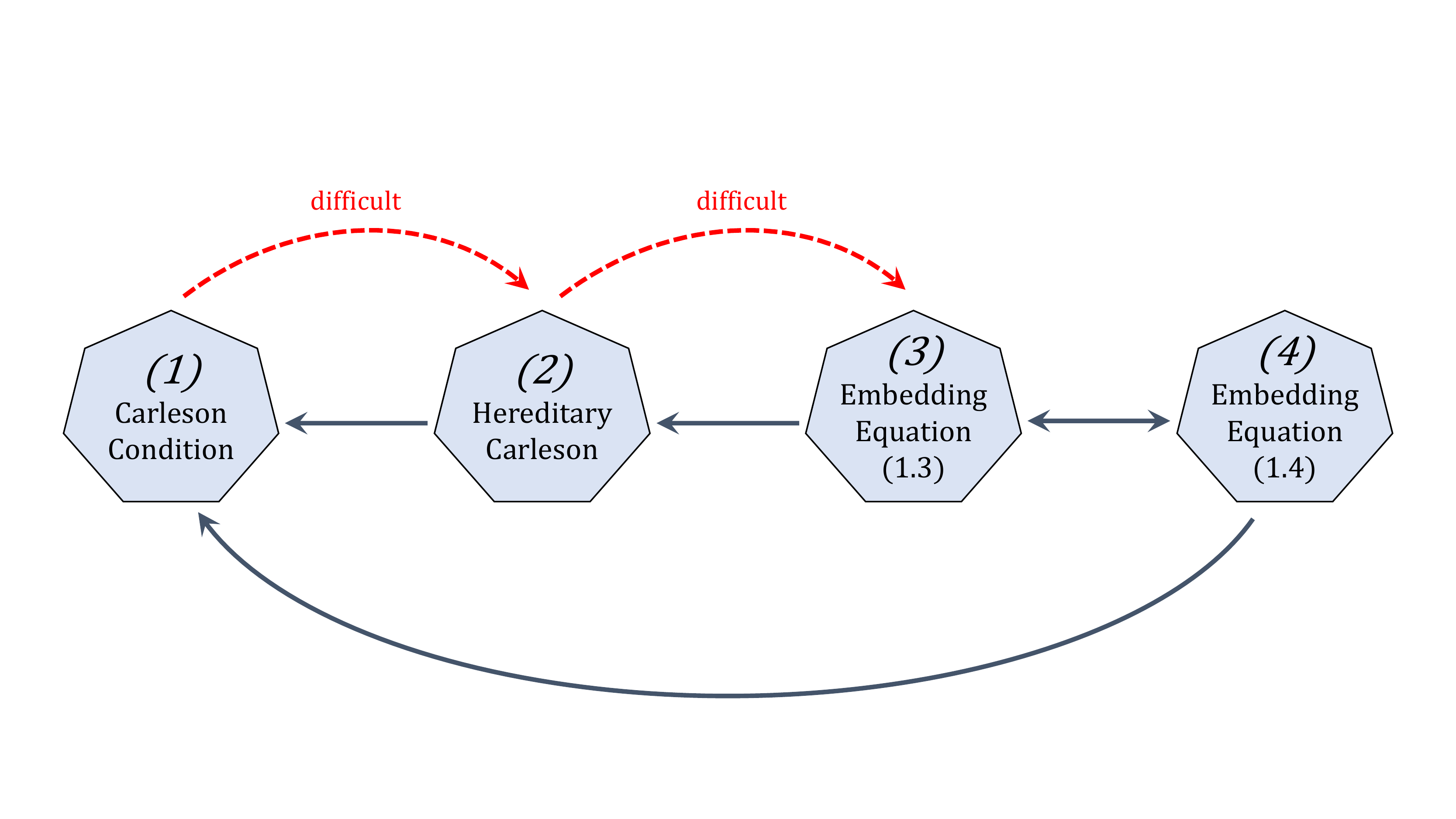}
      \caption{ }
  \end{subfigure}
  \begin{subfigure}[b]{0.45\linewidth}
    \includegraphics[width=\linewidth]{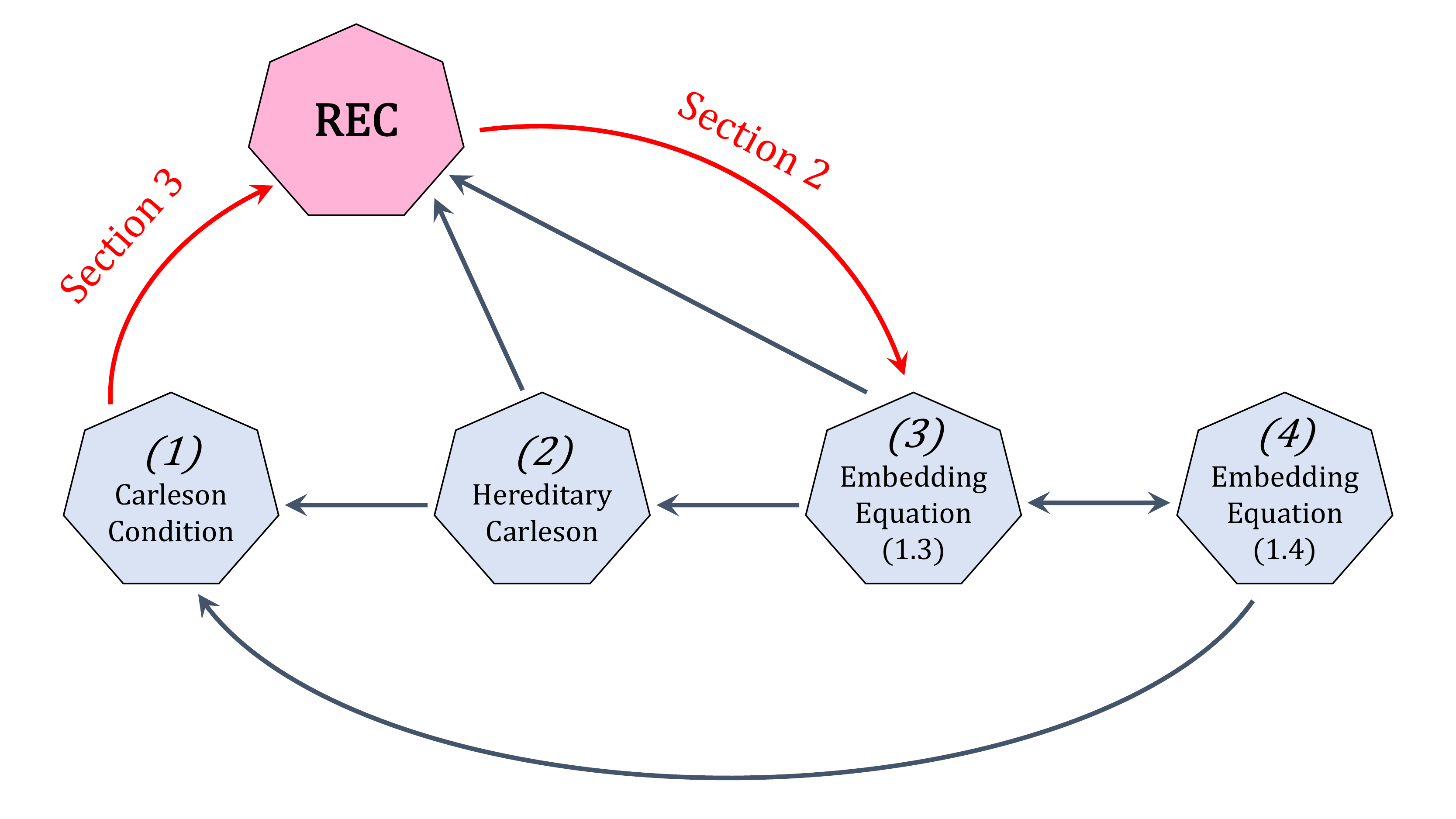}
    \caption{ }
  \end{subfigure}
  \caption{Structure of the Proof of Theorem \ref{main}. The dark blue arrows indicate implications which have been already obtained easily, and the red arrows indicate the difficult implications.}
  \label{fig:Summ}
\end{figure}

First, let us define the \textbf{potential} $\bV^\mu$ to be $\bI (\bI^*\mu)$. Then
	$$\bV^\mu(\al) = \sum_{\beta\geq\al} \mu(R_\beta), \:\: \forall \al\in\pd T^2.$$
Also, define the \textbf{energy}
	$$ \cE[\mu] \df\int  \bV^\mu\,d\mu= \sum_{\alpha\in T^2} \mu(R_\al)^2 = \|\bI^*\mu\|^2_{\ell^2(T^2)}.$$
Given a subset $E\subset \pd T^2$, we introduce
	$$ \cE_E[\mu] \df \sum_{R\in \cR_{E}} \mu(R)^2.$$
Of course the same definitions apply to  the simple tree $T$.

\begin{remark}
\label{cE}
Notice that $\cE[\mu|E]$ is considerably larger than $\cE_E[\mu]$:
	$$\cE[\mu|E] = \sum_{\al\in T^2} \mu(R_\al \cap U_E)^2 \gg \cE_E[\mu] = \sum_{\al:\: R_\al\subset U_E} \mu(R_\al)^2.$$ 
\end{remark}

\begin{defin}
\label{localEdf}
We say $\mu$ is a measure {\bf with restriction energy condition}, denoted $\mu\in REC$, provided that
	\begin{equation}
	\label{localE}
	\cE[\mu|E] \leq C \mu(E), \:\: \forall E\subset \pd T^2.
	\end{equation}
\end{defin}

Obviously, hereditary Carleson measures are REC -- this can be easily seen by taking $F=\pd T^2$ in \eqref{E:her2}.
Note that the Carleson condition may be written exactly as
	$$\text{Carleson condition: } \cE_E[\mu] \leq C \mu(E), \:\:\forall E\subset \pd T^2,$$
which is seemingly much weaker than restriction energy condition. In fact, these conditions are equivalent, 
as we show  below.

\begin{theorem}
\label{3equiv}
The Carleson condition, the restriction energy condition and the hereditary Carleson condition are all equivalent.
\end{theorem}


\subsection{Structure of the Paper} 

The big picture is summarized in Figure \ref{fig:Summ} (B):
first we are going to prove, in Section \ref{RECimp3}, the implication:

\begin{theorem}
\label{localEthm}
The restriction energy condition implies the embedding \eqref{bI}.
\end{theorem}

\begin{remark}
\label{localErk}
We already noticed the converse implication, see \eqref{rec1}, so this establishes the equivalence of restriction energy condition with embedding \eqref{bI}. Of course, Theorem \ref{localEthm} follows from \cite{AMPS}, but we wish to give  somewhat different proof avoiding the notion of capacity.
\end{remark}

The core result we will use for this is the following 

\begin{theorem}
\label{2D}
Let $\mu$ be a measure on $\pd T^2$ such that:
	\begin{itemize}
	\item $\bV^\mu \leq 1$ on $\supp(\mu)$;
	\item $\bV^\mu \geq \lambda$ on a set $F \subset \pd T^2$, for some large $\lambda$.
	\end{itemize}
Then there exists a positive function $\varphi$ on $T^2$ such that
	\begin{itemize}
	\item $\bI\varphi(\omega)\geq \lambda$, for all $\om \in F$, and
	\item $\|\varphi\|^2_{\ell^2(T^2)} \leq \frac{C}{\lambda}\cE[\mu]$. 
	\end{itemize}
\end{theorem}

Sections \ref{T1} and \ref{T2lemma} will be dedicated to proving this deeper result.

In Section \ref{CtoREC} we show that

\begin{theorem}
\label{CtoREC}
The bi-parameter Carleson condition implies the restricted energy condition.
\end{theorem}



\section{Embedding theorem for $REC$ measures}
\label{RECimp3}

We first need the following Lemma, which is the first place where we apply Theorem \ref{2D}.

\begin{lemma} \label{main_lm}
	Let $\mu, \rho$ be two $REC$ measures such that 
		\begin{itemize}
		\item $\bV^\mu \leq 1$ on $\supp(\mu)$, and
		\item $\bV^\mu \geq \lambda$ on $\supp(\rho)$.
		\end{itemize}
	Then
	\begin{equation} \label{main_ineq}
	|\rho| \le \frac{C}{\lambda^3} |\mu|\,.
	\end{equation}
\end{lemma}

\begin{proof}
By Theorem \ref{2D}, there is a function $\vf$ on $T^2$ such that
	$$\bI\vf(\om) \geq \lambda, \:\: \forall \om\in \supp(\rho)\:\:\text{ and }\:\:
		\|\vf\|^2_{\ell^2(T^2)} \leq \frac{C}{\lambda} \cE[\mu].$$
Then:
	$$\lambda |\rho| = \int \lambda\,d\rho \leq \int (\bI\vf)\,d\rho = \int \vf (\bI^*\rho) 
		\leq \|\vf\|_{\ell^2(T^2)} \cE[\rho]^{1/2} \leq C_\rho^{1/2} \|\vf\|_{\ell^2(T^2)} |\rho|^{1/2},$$
where we applied the $REC$ property in the last inequality. So
	$$|\rho| \leq \frac{C_\rho}{\lambda^2} \|\vf\|_{\ell^2(T^2)}^2 \leq \frac{C_\rho C}{\lambda^3} \cE[\mu] 
		\leq \frac{C_\rho C}{\lambda^3} |\mu|,$$
where the last inequality follows from the first assumption:
	$$\cE[\mu] = \int \bV^\mu\,d\mu \leq \int\,d\mu = |\mu|.$$
\end{proof}

\begin{theorem}
\label{murho}
Let $\mu, \rho$ be two $REC$ measures such that
$\bV^\mu\ge \lambda$ on $\supp \rho$. Then 
	$$ |\rho| \le C\frac{|\mu|}{\lambda^{7/3}}.$$
\end{theorem}

Remark that this result is similar to Lemma \ref{main_lm}, but much stronger as we are missing the boundedness assumption 
for $\bV^\mu$ on $\supp(\mu)$. We will get around this by constructing two new measures $\mu_1$, $\rho_1$ from $\mu$, $\rho$
to which Lemma \ref{main_lm} can be applied.

\begin{proof}
Let $C_\mu$, $C_\rho$ denote the $REC$ constants of $\mu$, $\rho$, respectively, and let
	$$F_\ep := \{ \om\in\pd T^2\: :\: \bV^\mu(\om) \geq \lambda^\ep\},$$
for some $\ep>0$ we will choose later to be $1/3$.

Split $\mu$ into
	$$\mu = \mu_0+\mu_1,$$
where
	$$\mu_0 := \mu|F_\ep.$$
We make some quick observations about these measures. First of all, obviously
	\begin{equation} \label{E:mu1bdd}
	\bV^{\mu_1} < \lambda^\ep \text{ on } \supp(\mu_1).
	\end{equation}
Observe that if we scale by $\lambda^\ep$ we have the first boundedness condition in Lemma \ref{main_lm}. So now we need to construct a complementary measure $\rho_1$ such that $\bV^{\mu_1} \geq c\lambda$ on $\supp(\rho_1)$.

Second of all, by Chebyshev,
	$$|\mu_0| = \mu(F_\ep) = \mu\{\bV^\mu \geq \lambda^\ep\} \leq \frac{1}{\lambda^\ep} \int \bV^\mu\,d\mu
		= \frac{1}{\lambda^\ep}\cE[\mu] \leq \frac{1}{\lambda^\ep} C_\mu |\mu|,$$
so
	\begin{equation}\label{Ch}
	|\mu_0| \leq C_\mu \frac{|\mu|}{\lambda^\ep}.
	\end{equation}
Third, it is easy to see that both $\mu_0$ and $\mu_1$ are also $REC$ measures with constant $C_\mu$.
Finally,
	\begin{equation}\label{E:Vmu0bd}
	\int \bV^{\mu_0}\,d\rho \leq (C_\mu C\rho |\mu_0| |\rho|)^{1/2}.
	\end{equation}
To see this:
	$$ \int\bV^{\mu_0}\,d\rho = \int \bI (\bI^*\mu_0)\,d\rho = \int (\bI^*\mu_0)(\bI^*\rho)
	\leq \|\bI^*\mu_0\|_{\ell^2(T^2)} \|\bI^*\rho\|_{\ell^2(T^2)}
		= (\cE[\mu_0]\cE[\rho])^{1/2},$$
and apply the $REC$ property to the last terms.

Returning now to $\rho$, we will say that $\rho$ is ``good'' if
	\begin{equation}\label{E:rhogood}
	|\rho| \leq K \frac{|\mu|}{\lambda^{2+\ep}},
	\end{equation}
for some large constant $K$ we will choose later. Suppose this is not the case though, and combine the ``badness'' of $\rho$
with \eqref{Ch} to obtain:
	$$|\mu_0| \leq C_\mu \frac{|\mu|}{\lambda^\ep} \leq C_\mu \frac{1}{\lambda^\ep} \frac{\lambda^{2+\ep}}{K}|\rho|
		= \frac{C_\mu}{K}\lambda^2|\rho|.$$
Then \eqref{E:Vmu0bd} gives us
	\begin{equation}\label{E:Vmu0bd2}
	\int \bV^{\mu_0}\,d\rho \leq \frac{C_\mu\sqrt{C_\rho}}{\sqrt{K}}\lambda|\rho|.
	\end{equation}

Consider now a fixed, very small $\delta>0$. Again by Chebyshev:
	$$\rho\{\bV^{\mu_0} > \delta\lambda\} \leq \frac{1}{\delta\lambda} \int \bV^{\mu_0}\,d\rho
		\stackrel{\eqref{E:Vmu0bd2}}{\leq} \frac{C_\mu\sqrt{C_\rho}}{\delta\sqrt{K}}|\rho|.$$
Keep in mind now that, by assumption,
 $\supp(\rho) \subset \{\bV^\mu \geq \lambda\} = \{\bV^{\mu_0} +\bV^{\mu_1} \geq\lambda\}$, and so
 	$$\rho\{\bV^{\mu_0} > \delta\lambda\} = |\rho| - \rho\{\bV^{\mu_0} \leq \delta\lambda\}
		\geq |\rho| - \rho\{\bV^{\mu_1} \geq (1-\delta)\lambda\},$$
which, combined with the estimate above, leads us to
	\begin{equation}\label{E:rhobd}
	\rho\{\bV^{\mu_1} \geq (1-\delta)\lambda\} \geq |\rho| \bigg( 1 - \frac{C_\mu\sqrt{C_\rho}}{\delta\sqrt{K}}\bigg) \geq |\rho|(1-\delta),
	\end{equation}
where the last inequality comes from us finally choosing $K$ large enough so that
	$$\frac{C_\mu\sqrt{C_\rho}}{\delta\sqrt{K}} \leq \delta.$$
For instance, if $\delta=0.01$, the relationship in \eqref{E:rhobd} becomes ``on 99\% of $\rho$ we have $\bV^{\mu_1} \geq 0.99\lambda$.''

Call now $F := \{\bV^{\mu_1} \geq (1-\delta)\lambda\}$ and set
	$$\rho_1 := \rho|F,$$
so $\rho_1$ is another $REC$ measure of basically the same mass as the original $\rho$, but it has the right relationship to $\bV^{\mu_1}$ to allow for Lemma \ref{main_lm}. Specifically, if we also let
	$$\kappa := \frac{\mu_1}{\lambda^\ep},$$
then $\kappa$ and $\rho_1$ are two $REC$ measures such that
	\begin{itemize}
	\item $\bV^\kappa \leq 1 $ on $\supp(\kappa)$ by \eqref{E:mu1bdd}, and
	\item $\bV^\kappa \geq \frac{1-\delta}{\lambda^{\ep-1}}$ on $\supp(\rho_1)$ by \eqref{E:rhobd}.
	\end{itemize}
By Lemma \ref{main_lm} we have
	$$|\rho_1| \leq \frac{C|\mu|}{\lambda^{3-2\ep}}.$$
Now note that \eqref{E:rhobd} translates to $|\rho_1| = \rho(F) \geq |\rho|(1-\delta)$, so
	\begin{equation}\label{E:rhobad}
	|\rho| \leq \frac{1}{1-\delta}|\rho_1| \leq \frac{1}{1-\delta} \frac{C|\mu|}{\lambda^{3-2\ep}}.
	\end{equation}
	
So, there are two possibilities: either $\rho$ is good, in which case \eqref{E:rhogood} holds, or $\rho$ is bad and then \eqref{E:rhobad} holds. If we set $2+\ep = 3-2\ep$, so $\ep=1/3$, then either estimate will yield the desired result.
\end{proof}


\subsection{Mutual energy of pieces of $REC$ measures}
\label{mutual}

Let $\mu$ be an $REC$ measure and let $F\subset E\subset \pd T^2$. Below is the trivial estimate of the mutual energy: 
	\begin{eqnarray*}
	\int \bV^{\mu|E} d\mu|F &=& \int (\bI^*(\mu|E))\cdot  (\bI^*(\mu|F))\\
			&\leq& \|\bI^*(\mu|E)\|_{\ell^2(T^2)} \: \|\bI^*(\mu|F)\|_{\ell^2(T^2)}
				= \big(\cE[\mu|E]\cE[\mu|F]\big)^{1/2}\\
			&\leq& C_\mu \big(\mu(E)\mu(F)\big)^{1/2}
	\end{eqnarray*}
Here is the improvement.

\begin{theorem}
\label{impr}
Let $\mu$ be $REC$ measure, and let $F\subset E\subset T^2$. Then
$$
\int \bV^{\mu|E} d\mu|F \le C\mu(E)^{3/7}\mu(F)^{4/7}\,.
$$
\end{theorem}

\begin{proof} 
Let $k\ge 0$. Denote $F_k\df\{ x\in F: \bV^{\mu|E} \in [2^k, 2^{k+1})\}$. Then by Theorem \ref{murho} we have
$$
\mu(F_k) \le 2^{-7/3k} \mu(E)\,.
$$
Put $\lambda\df \frac{\int \bV^{\mu|E} d\mu|F }{\mu(F)}$.  If $\la\le 1$ then theorem follows trivially with $C=1$. Let $\lambda\in [2^j, 2^{j+1}), j\ge 0$. Repeating Lemma 2 of \cite{AMPS}  and using this display formula, we get
$$
\lambda\mu(F) \le \frac{\lambda}{2} \mu(F) + C\lambda^{-4/3} \mu(E)\Rightarrow \lambda^{7/3} \le C\frac{\mu(E)}{\mu(F)}.
$$
This gives the claim of the theorem.
\end{proof}


\subsection{Embedding theorem for $REC$ measures}
\label{emb}

\begin{proof}[Proof of Theorem \ref{localEthm}]
We start almost exactly as in \cite{AMPS}. We write
$$
E_k\df \{ \alpha\in \pd T^2: \bI\vf (\alpha) \ge 2^k\}.
$$
Unlike \cite{AMPS} we put $\mu_k\df \mu|E_k$.
Then
\begin{align*}
\int |\bI \vf|^2 d \mu & \lesssim \sum 2^{2k} |\mu_k| \le \sum 2^k \int( \bI\vf )\,d\mu_k = \sum 2^k \int \vf (\bI^*\mu_k)
\\
&\le \|\vf\|_{\ell^2(T^2)} \|\sum 2^k \bI^*\mu_k\|_{\ell^2(T^2)}.
\end{align*}
Expanding the square in $ \|\sum 2^k \bI^*\mu_k\|_{\ell^2(T^2)}^2$ we get $\sum_k\sum_{j\le k} 2^{j+k} \int \bV^{\mu_k}\, d\mu_j$. Consider the diagonal part $\sum_k 2^{2k} \int \bV^{\mu_k}\, d\mu_k=\sum_k 2^{2k}\cE[\mu|E_k] \le \sum_k 2^{2k} \mu(E_k)$. The last inequality uses  exactly $REC$ property. Thus the diagonal part is $\big(\int (\bI\vf)^2\, d\mu\big)^{1/2}$.

We are left to prove that off-diagonal part is $\lesssim \sum_k 2^{2k}\mu(E_k)$ as well. Here we follow \cite{AH}, \cite{AMPS}.
Using Theorem \ref{impr} we can write
\begin{align*}
\sum_k\sum_{j< k} 2^{j+k} \int \bV^{\mu_k}\, d\mu_j & \lesssim \sum_k 2^k |\mu_k|^{4/7}\sum_{j\le k} 2^j\,|\mu_j|^{3/7}
\\
& =  \sum_k 2^{8k/7 }|\mu_k|^{4/7}2^{-k/7}\sum_{j\le k} 2^j\,|\mu_j|^{3/7} 
\\
& \le \big(\sum 2^{2k} |\mu_k|\big)^{4/7} \bigg(\sum_k 2^{-k/3} \big(\sum_{j\le k} 2^j\,|\mu_j|^{3/7}\big)^{7/3}\bigg)^{3/7}.
\end{align*}

Now
\begin{align*}
\big(\sum_{j\le k} 2^j\,|\mu_j|^{3/7}\big)^{7/3} & \le \bigg(\sum_{j\le k} 2^{j/14} \, 2^{13j/14}\,|\mu_j|^{3/7}\bigg)^{7/3} 
\\
&\left(\bigg(\sum_{j\le k} (2^{j/14})^{4/7}\bigg)^{7/4}\cdot \bigg(\sum_{j\le k} (2^{13j/14}\,|\mu_j|^{3/7})^{7/3}\bigg)^{3/7}\right)^{7/3}
\\
&\lesssim \left( 2^{k/14}\bigg(\sum_{j\le k} 2^{13j/6} |\mu_j|\bigg)^{3/7}\right)^{7/3} = 2^{k/6} \bigg(\sum_{j\le k} 2^{13j/6} |\mu_j|\bigg).
\end{align*}
Combining this with the previous display formula, we get
\begin{align*}
\sum_k\sum_{j< k} 2^{j+k} \int \bV^{\mu_k}\, d\mu_j & \lesssim  \big(\sum 2^{2k} |\mu_k|\big)^{4/7}\bigg(\sum_k 2^{-k/6}\sum_{j\le k} 2^{13j/6} |\mu_j|\bigg)^{3/7}
\\
& \lesssim \big(\sum 2^{2k} |\mu_k|\big)^{4/7} \big(\sum 2^{2k} |\mu_k|\big)^{3/7}= \sum 2^{2k} |\mu_k|\,.
\end{align*}

\end{proof}





\section{Proof that the Bi-parameter Carleson condition implies REC}
\label{CtoREC}

\begin{proof}[Proof of Theorem \ref{CtoREC}]
We assume the bi-parameter Carleson condition:
\begin{equation}
\label{e:511}
\cE_E[\mu]= \sum_{R\in\mathcal{R}_E} \mu(R)^2 \leq \mu(E),\quad \forall\; E\subset \pd T^2.
\end{equation}
But let $F_0$ be a subset of $\pd T^2$ such that for $\mu_0\df \mu|F_0$ the following holds with a large constant $C$.

 \begin{equation}
 \label{e:55}
\mathcal{E}[\mu_0] = C|\mu_0|.
\end{equation}
Moreover, we can assume that 
\begin{equation}
\label{maxC}
C=\max_F\frac{\cE[\mu|F]}{|\mu(F)|}\,.
\end{equation}
This is because we assumed in Section \ref{intro} that $T^2$ is a finite graph (albeit a very large one).

\medskip

The main ingredient in estimating $C$ is provided by Lemma \ref{cEcE} of Section \ref{pr32}.
We are going to prove that
\begin{equation}
\label{C69}
C \le A\, C_{\ref{lemma32}}^2,
\end{equation}
where $C_{\ref{lemma32}}$ is the constant from Lemma \ref{cEcE} of Section \ref{pr32}.
\subsubsection{Part 1: making $\mu_0$ to be almost equilibrium} We start by introducing some additional notation. Given a set $E\subset \pd T^2$ and a measure $\nu$ we defined above the local energy of $\nu$ at $E$
\[
\mathcal{E}_E[\nu] := \sum_{R\in\mathcal{R}_E}\nu(R)^2.
\]
In particular, we have $\mathcal{E}_{Q_0}[\nu] = \mathcal{E}[\nu]$.

Now we have \eqref{e:55}, hence
\[
\int \mathbb{V}^{\mu_0}\,d\mu_0 =  \int C\,d\mu_0,
\]
which, as we will now see,  means that $\mathbb{V}^{\mu_0} \geq \frac{C}{3}$ on a major part of $\supp\mu_0$. For now we want to get rid of those points in $\supp\mu_0$ where the potential is not large enough whilst conserving the total energy. We do so by the power of the following lemma
\begin{lemma}\label{l:l3.1}
Assume that $\nu$ is a non-negative measure on $Q_0$, $\supp\nu = E\subset Q_0$ and
\[
\mathcal{E}[\nu] = \int \mathbb{V}^{\nu}\,d\nu \geq C\nu(E) = C|\nu|.
\]
Then there exists a set $\tilde{E} \subset E$ such that
\[
\mathbb{V}^{\tilde{\nu}} \geq \frac{C}{3},\quad \textup{on}\; \tilde{E},
\]
and 
\[
\mathcal{E}[\tilde{\nu}] \geq \frac{1}{6}\mathcal{E}[\nu].
\]
Here $\tilde{\nu} := \nu\vert_{\tilde{E}}$.
\end{lemma}

\begin{proof}
First we assume that $C = 3$ (otherwise we just rescale).
Let $E_0 := \{t\in E: \mathbb{V}^{\nu} \leq1\}$ and $\sigma_0 := \nu\vert_{E_0}$. Assume we have constructed $\sigma_j,\; j=0, \dots ,k-1$, and the sets $E_j$. We then define $E_k$ to be
\[
E_k = \{\omega\in E\setminus\bigcup_{j=0}^{k-1}E_j: \mathbb{V}^{\nu-\sum_{j=0}^{k-1}\sigma_j}(\omega) \leq 1\},
\]  
and we let $\sigma_k = \nu\vert_{E_k}$.

Since $T^2$ is finite, the procedure must stop at some (possibly very large) number $N$, i.e. $E_j = \emptyset$ for $j>N$. We let $E_{\infty} = E\setminus\bigcup_{j=0}^{N}E_j$ (we do not know yet if this set is non-empty), $\sigma_{\infty} = \nu\vert_{E_{\infty}}$. By construction we have
\[
\mathbb{V}^{\sigma_{\infty}}(\omega) \geq 1,\quad \omega \in E_{\infty}.
\]
Next we compute the energy of $\nu$,
\begin{equation}\notag
\begin{split}
&\mathcal{E}[\nu] = \int\mathbb{V}^{\nu}\,d\nu = \sum_{N\geq j\geq0}\sum_{N\geq k\geq0}\int \mathbb{V}^{\sigma_j}\,d\sigma_k =
2\sum_{N\geq k\geq0}\sum_{N\geq j\geq k}\int \mathbb{V}^{\sigma_j}\,d\sigma_k \\
&=2\sum_{N>k\geq0}\int\mathbb{V}^{\sum_{N\geq j\geq k}\sigma_j}\,d\sigma_k + 2\int \mathbb{V}^{\sigma_{\infty}}\,d\sigma_{\infty} 
\\
&=2\sum_{N>k\geq0}\int\mathbb{V}^{\nu-\sum_{j=0}^{k-1}\sigma_j}\,d\sigma_k + 2\int \mathbb{V}^{\sigma_{\infty}}\,d\sigma_{\infty}\\
&\leq 2\sum_{k=0}^{N}\int\,d\sigma_k + 2\mathcal{E}[\sigma_{\infty}] = 2|\nu| + 2\mathcal{E}[\sigma_{\infty}].
\end{split}
\end{equation}
Since $\mathcal{E}[\nu] \geq 3|\nu|$ by assumption, we have
\[
\mathcal{E}[\sigma_{\infty}] \geq \frac16\mathcal{E}[\nu],
\]
it remains to let $\tilde{\nu} := \sigma_{\infty}$, $\tilde{E} := E_{\infty}$, and we are done.
\end{proof}
We apply this lemma to $\mu_0$ and $F_0$ (we remind that $\mu_0 = \mu\vert_{F_0}$, so that $\supp\mu_0\subset F_0$) obtaining the set $F_{1}\subset F_0$ and a measure $\mu_{1} = \mu_0\vert_{F_{1}} = \mu\vert_{F_{1}}$ that satisfies $\mathbb{V}^{\mu_{1}} \geq \frac{C}{3}$ on $F_{1}$ and 
$$
\mathcal{E}[\mu_{1}] \geq \frac16\mathcal{E}[\mu_0]\,.
$$
 Finally we let
\begin{equation}\label{e:57}
E := \left\{t\in Q_0:\; \mathbb{V}^{\mu_{1}}(t)\geq \frac{1}{4 C_{\ref{lemma32}}}\cdot\frac{C}{3}\right\}\,.
\end{equation}

\subsubsection{Part 2: Why is $E$ the right set to consider? Main Lemma \ref{l:l2}.}
First we state another lemma that allows us to estimate the total energy of an almost equilibrium measure by its local energy at a certain level set. This is the main ingredient of proving that bi-parameter Carleson condition implies REC condition.
For the proof  of the following lemma see \cite{AMPS} and Lemma \ref{cEcE} of  Section \ref{pr32} below.
\begin{lemma}
\label{l:l2}
Let $\nu\geq0$ be a measure on $Q_0$ such that
\[
\mathbb{V}^{\nu}(t) \geq C_1,\quad t\in \supp\nu,
\]
and let 
\[
E := \left\{t\in Q_0:\; \mathbb{V}^{\nu}(t) \geq \frac{C_1}{4 C_{\ref{lemma32}}}\right\}.
\]
Then
\begin{equation}\label{e:59}
\mathcal{E}_E[\nu] \geq \frac{1}{2}\mathcal{E}[\nu].
\end{equation}
\end{lemma}

\medskip


We already mentioned that this lemma will be proved in Section \ref{pr32}. Now we are going to use it.
Let $\sigma: = \mu\vert_{E}$. By definition,
\begin{equation}
\label{C3dot4}
\frac{C}{3\cdot4 C_{\ref{lemma32}}} |\sigma| \le \int\bV^{\mu_1} \, d\sigma\,.
\end{equation}
On the other hand,
$$
\int\bV^{\mu_1} \, d\sigma \le \cE[\mu_1]^{1/2} \cE[\sigma]^{1/2} \le \cE[\mu_1]^{1/2}\cE[ \mu\vert_{E}]^{1/2} \le 2^{1/2}\cE_E[\mu_1]^{1/2}\cE[ \mu\vert_{E}]^{1/2} ,
$$
where the last inequality follows from Lemma \ref{l:l2} applied to $\nu=\mu_1$ and $C_1= C/3$.

Now by assumption \eqref{e:511} we have $\cE_E[\mu_1]^{1/2} \le |\mu|F_1| =\mu(F_1) \le  \mu(E) = |\sigma|$. By  assumption \eqref{maxC} of maximality we have
$\cE[ \mu\vert_{E}] \le C \mu(E) = C|\sigma|$. Let us combine this with the last display formula to get
$$
\int\bV^{\mu_1} \, d\sigma \le 2^{1/2} C^{1/2} |\sigma|\,.
$$
Combine that with \eqref{C3dot4} to obtain
$$
\frac{C}{3\cdot4 C_{\ref{lemma32}}}\le 2^{1/2} C^{1/2},
$$
which gives us $C\le 288 C_{\ref{lemma32}}^2$, which we wanted to prove, as $C=\max_F\frac{\cE[\mu|F]}{|\mu(F)|}$.


\end{proof}

\bigskip

In the next section  we start to prepare the proof of Theorem \ref{2D}, whose proof will be finished in Section \ref{T2lemma}.

\section{Lemma on majorization with small energy. A case of ordinary tree}
\label{T1}

All trees below can be very deep, but it is convenient to think that they are finite.
Estimates will not depend on the depth.

First, some notation. For every dyadic interval $J$, we call:

\begin{minipage}{0.5\textwidth}\raggedleft
	\begin{itemize}[leftmargin=0in]
	\item $Q_J$ -- the square with base $J$;
	\item $Top_J$ -- the top half (rectangle) of $Q_J$.
	\end{itemize}
\end{minipage}
\hfill
\noindent\begin{minipage}{0.4\textwidth}
\includegraphics[width=\linewidth]{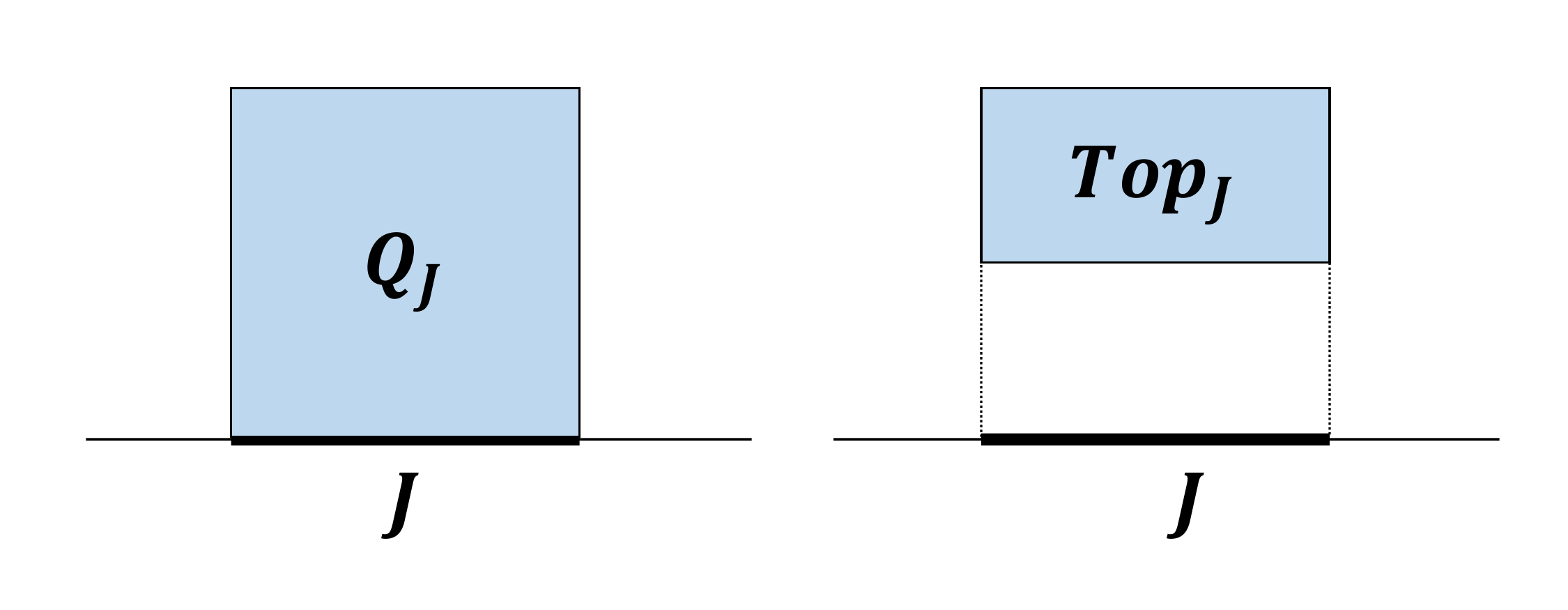}
\end{minipage}%

Let $I_0=[0,1]$ and identify the dyadic intervals in $\mathcal{D}(I_0)$ with vertices of the tree $T:= \mathcal{T}(I_0)$, as before. Let $\cS$ be a family of disjoint dyadic subintervals of $I_0$, and define:
	$$W(\cS) := \bigcup_{J\in\cS}Q_J,$$
	$$O(S) := Q_{I_0}\setminus W(\cS)=\cup_{\ell,\, Top_\ell \, \text{not in}\, Q_J, J\in S} Top_\ell.$$
To visualize these sets, one may think of the dyadic tree in the usual way, as in Figure \ref{fig:Pic2} (A), but in this section it may be more useful to identify each $J \in \mathcal{D}(I_0)$ with the rectangle $Top_J$, as in Figure \ref{fig:Pic2} (B).

\begin{figure}[h!]
  \centering
  \begin{subfigure}[b]{0.45\linewidth}
    \includegraphics[width=\linewidth]{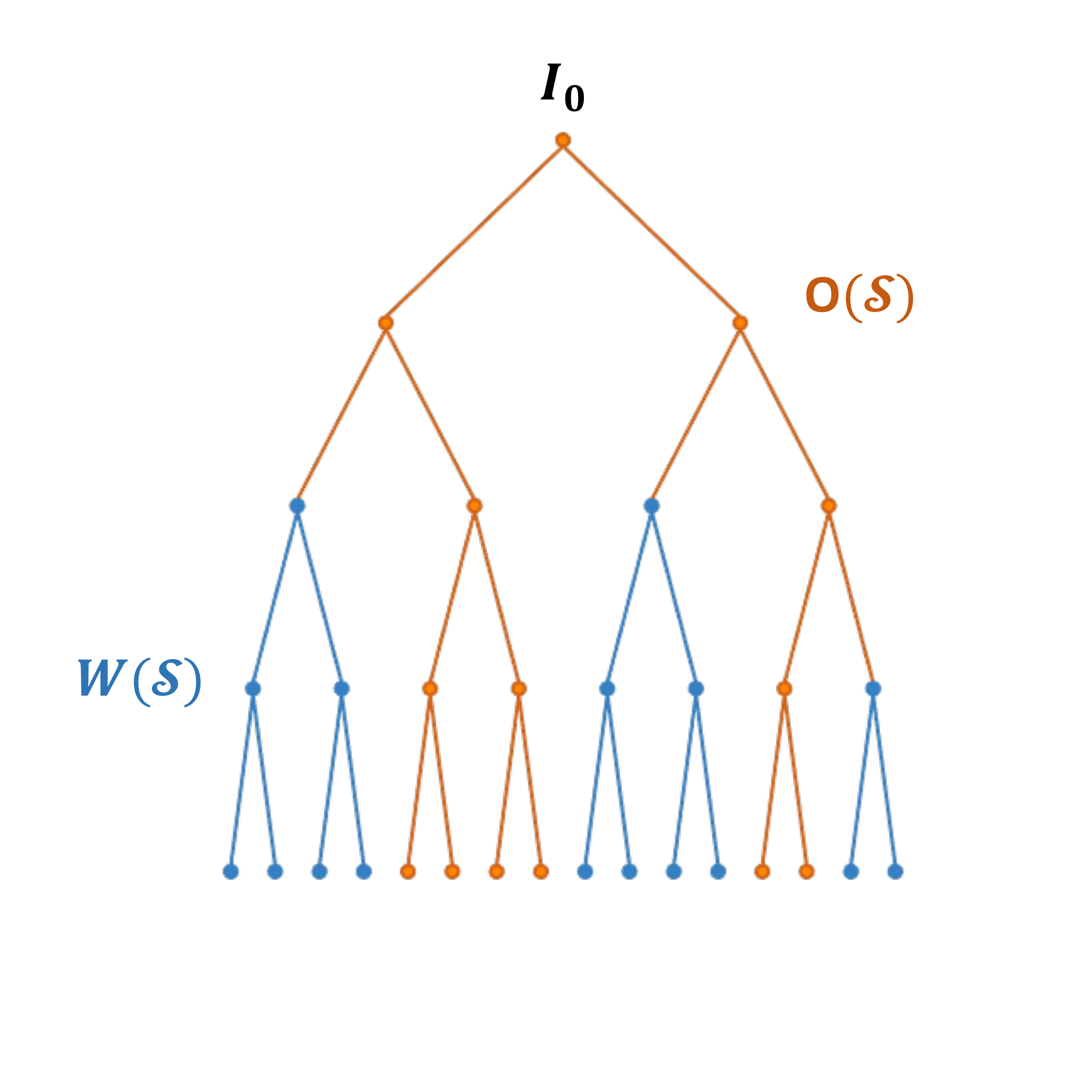}
      \caption{ }
  \end{subfigure}
  \begin{subfigure}[b]{0.45\linewidth}
    \includegraphics[width=\linewidth]{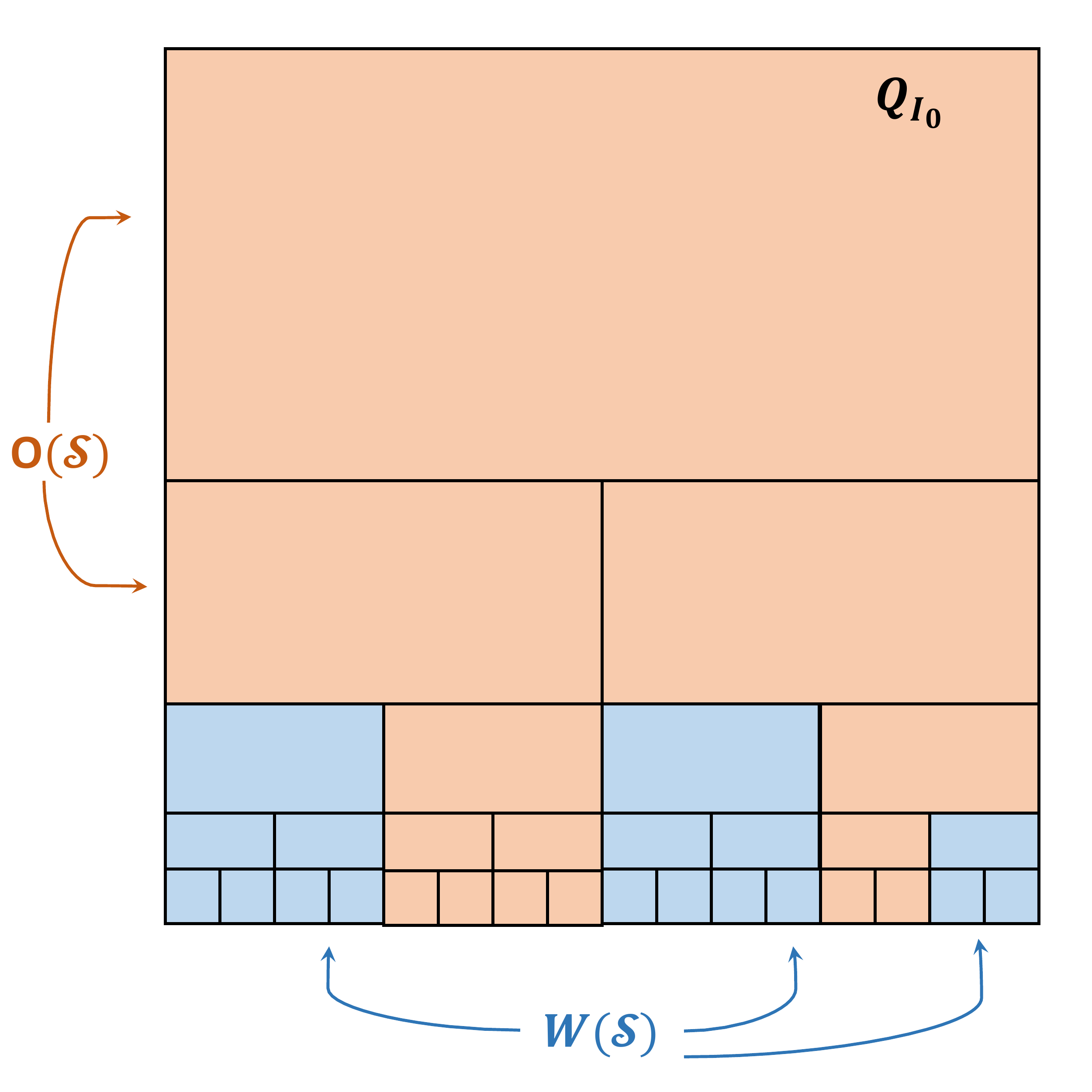}
    \caption{ }
  \end{subfigure}
  \caption{Dyadic tree and the sets $W(\cS)$ and $O(\cS)$.}
  \label{fig:Pic2}
\end{figure}

For vertices $\alpha$ of the tree $T$, we write $\alpha \in W(\cS)$ if there is a $\beta \in \cS$ such that $\alpha \leq \beta$.
Given a measure $\sigma$ on $T$, define:
	\begin{equation*}
	I_{\cS}^*\sigma(\alpha) := \left\{ \begin{array}{ll}
			I^*\sigma(\alpha), & \text{if } \alpha \in W(\cS)\\
			0, & \text{if } \alpha \in O(\cS).
		\end{array}\right.
	\end{equation*}
and the local potential:
	\begin{equation*}
	V_{\cS}^{\sigma}(\alpha) := \left\{\begin{array}{ll}
		\sum_{\alpha':\beta\geq\alpha'\geq\alpha} I^*\sigma(\alpha'), & \text{if } \alpha \leq \beta,\:\beta\in\cS \\
		0, & \text{if } \alpha \in O(\cS).
		\end{array}\right.
	\end{equation*}
Then  we conveniently have
	\begin{equation} \label{ell2}
	\sum I^*_{\cS}\sigma (\alpha)^2 = \int V^{\sigma}_{\cS} d\sigma\,.
	\end{equation}

\begin{lemma} \label{Phi}
Let $\sigma$ be a measure on $\partial T$ and $\cS$ be a collection of disjoint dyadic subintervals of $I_0$ satisfying $O(\cS)\neq\emptyset$.
Let $f\geq 0$ be a function on the tree $T$ such that $f=0$ on $W(\cS)$. Let $F\subset\partial T \cap W(\cS)$, and suppose that for a large $\lambda >>1$, the potential $V^{\sigma}$ satisfies:
	\begin{equation} \label{gela}
	V^{\sigma}(\omega)\geq\lambda,\quad \omega \in F\,,
	\end{equation}
and
	\begin{equation} 
	\label{le1}
	V^{\sigma}(\alpha) <1, \quad \alpha \in O(\cS)\,.
	\end{equation}
Then there exists another function $\Phi$ on $T$ such that, with positive absolute constants $c$, $C$:
	\begin{equation} \label{geIf}
	I\Phi(\omega) \geq c\, If(\omega),\quad \omega \in F\,,
	\end{equation}
and
	\begin{equation} \label{enest1}
	\|\Phi\|_{\ell^2(T)} \leq \frac{C}{\lambda} \|f\|_{\ell^2(T)}\,.
	\end{equation}
\end{lemma}

\begin{proof}

We will give a formula for $\Phi$. This function will be zero on $O(\cS)$ -- see Figure \ref{fig:Pic3}(A) -- and on $W(\cS)$ it is defined as follows: if $\alpha \leq \beta$ for some $\beta \in \cS$, then

\begin{equation}
\label{PhiDef}
\Phi(\alpha) := \left\{\begin{array}{ll}
	\frac1{\lambda}If(\beta)\,I^*_{\cS}\sigma(\alpha), & \text{if } \,\,\sum_{\beta\geq \alpha' \geq \alpha}I^*_{\cS}\sigma(\alpha') \leq \lambda\\
	0, & \text{if }\,\,\sum_{\beta\geq \alpha' \geq \alpha}I^*_{\cS}\sigma(\alpha') > \lambda.  
	\end{array}\right.
\end{equation}

We prove first \eqref{geIf}. Let $\omega \in F$ and let $\beta\in\cS$ such that $\omega\leq \beta$. Since $f=0$ on $W(\cS)$,
	\begin{equation}
	\label{E:If}
	If(\omega) = \sum_{\alpha\geq\omega} f(\alpha) = \sum_{\alpha\geq\beta} f(\alpha) = If(\beta).
	\end{equation}
For $\Phi(\omega)$, we have two cases.

\textbf{Case 1:} $\Phi(\omega) = 0$.  Notice that the case $ I f(\omega)=0$, $\om\in F$, is then done: obviously for $\om\le\beta$, $\om \in F$, $I\Phi(\omega) \ge 0= I f (\omega)$. 

But $If(\om) =I f(\beta)$, see \eqref{E:If}. So without loss of generality we can think below that $If(\beta)>0$.

Let $\om \in F$ and
 let $\gamma$ be the largest $\omega\leq\gamma\leq\beta$ such that 
 $\Phi(\gamma) = 0$; see Figure \ref{fig:Pic3}(B). Remark that $\gamma<\beta$, that is we cannot have 
 $\Phi(\alpha)=0$, $\forall \omega\leq\alpha\leq\beta$.  Let us explain that.
 
Recall that $If(\beta)>0$. Since $\Phi(\beta) =\frac1{\lambda}If(\beta)\,I^*_{\cS}\sigma(\beta)$ then  the  first of reasons why  $\Phi(\beta) = 0$ is $I^*_{\cS}\sigma(\beta) =0$. In other words $\sigma(\beta)=0$ (since $\sigma$ is measure only on the boundary of the tree).
The second reason is (see definition \eqref{PhiDef})
\begin{equation}
\label{secondR}
 I^*_{\cS}\sigma(\beta) > \lambda\,.
 \end{equation}
 Let us bring the first reason to contradiction.

For $\om\le\beta$, $\om \in F$ we know that $V^\sigma(\om)\ge \la$. Notice that  $V^\sigma(\om) = V^\sigma(\beta)$ if  $\sigma(\beta)=0$. Thus, we have $V^\sigma(\beta) \ge \la$, but we also have $V^\sigma(\hat\beta)<1$ by assumption \eqref{le1}. So 

$$
\sigma(\beta)\ge \la-1\ge \frac{\la}{2}.
$$
 But this is impossible: we just wrote that $\sigma(\beta) =0$.  This is a contradiction.
 
 \medskip
 
 Notice that it follows from the assumption that $O(\cS)\neq\emptyset$ that $I_0=root_T\in O(\cS)$, 
 which gives the following mass estimate for $\sigma$:
	\begin{equation}
	\label{massroot}
	\|\sigma\| = II^*\sigma(I_0) = V^{\sigma}(I_0) < 1,
	\end{equation}
by \eqref{le1}. But this means that
	\begin{equation} 
	\label{mass1}
	\sigma(\alpha)=I^*\sigma(\alpha)<1<\frac{\la}{2}, \quad\forall \alpha \in T.
	\end{equation}

\medskip

So, if $\Phi(\beta) = 0$, then by definition of $\Phi$ (see \eqref{PhiDef}) we would have  only the second possibility left: \eqref{secondR}, namely,  this may happen only if $\sigma(\beta)=I^*\sigma(\beta) =I^*_{\cS}\sigma(\beta)>\lambda>1$, a contradiction with \eqref{mass1}. So the second reason for $\Phi(\beta)$ to be zero is disproved as well.

Note also that, once $\Phi(\alpha)\neq 0$, then $\Phi(\alpha')\neq 0$ for all $\alpha\leq\alpha'\leq\beta$:
	$$\sum_{\beta\geq\alpha''\geq\alpha'}I^*\sigma(\alpha'') \leq \sum_{\beta\geq\alpha''\geq\alpha}I^*\sigma(\alpha'') \leq \lambda
	\text{, so } \Phi(\alpha')\neq 0.$$

\begin{figure}[h!]
  \centering
  \begin{subfigure}[b]{0.45\linewidth}
    \includegraphics[width=\linewidth]{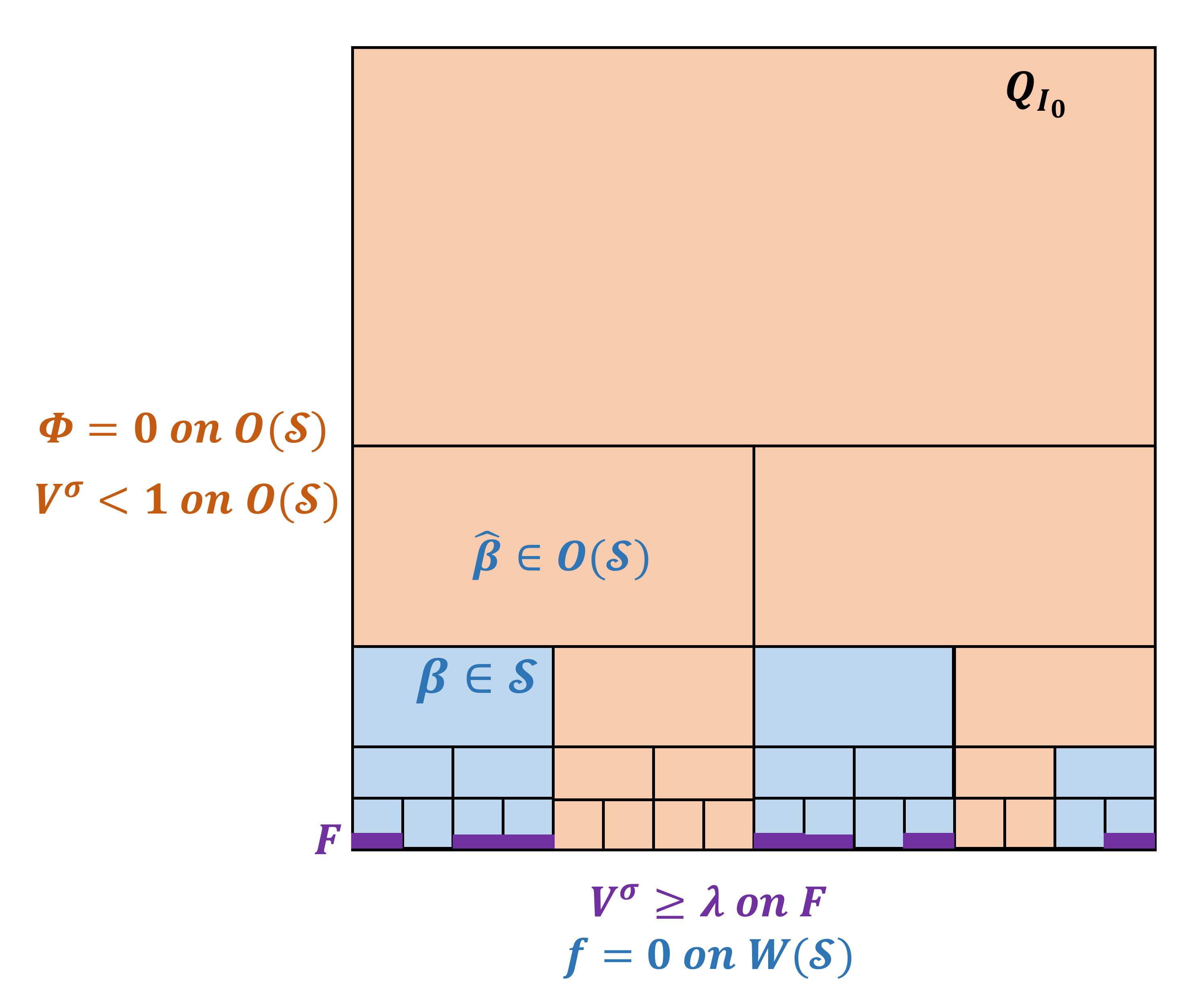}
      \caption{ }
  \end{subfigure}
  \begin{subfigure}[b]{0.45\linewidth}
    \includegraphics[width=\linewidth]{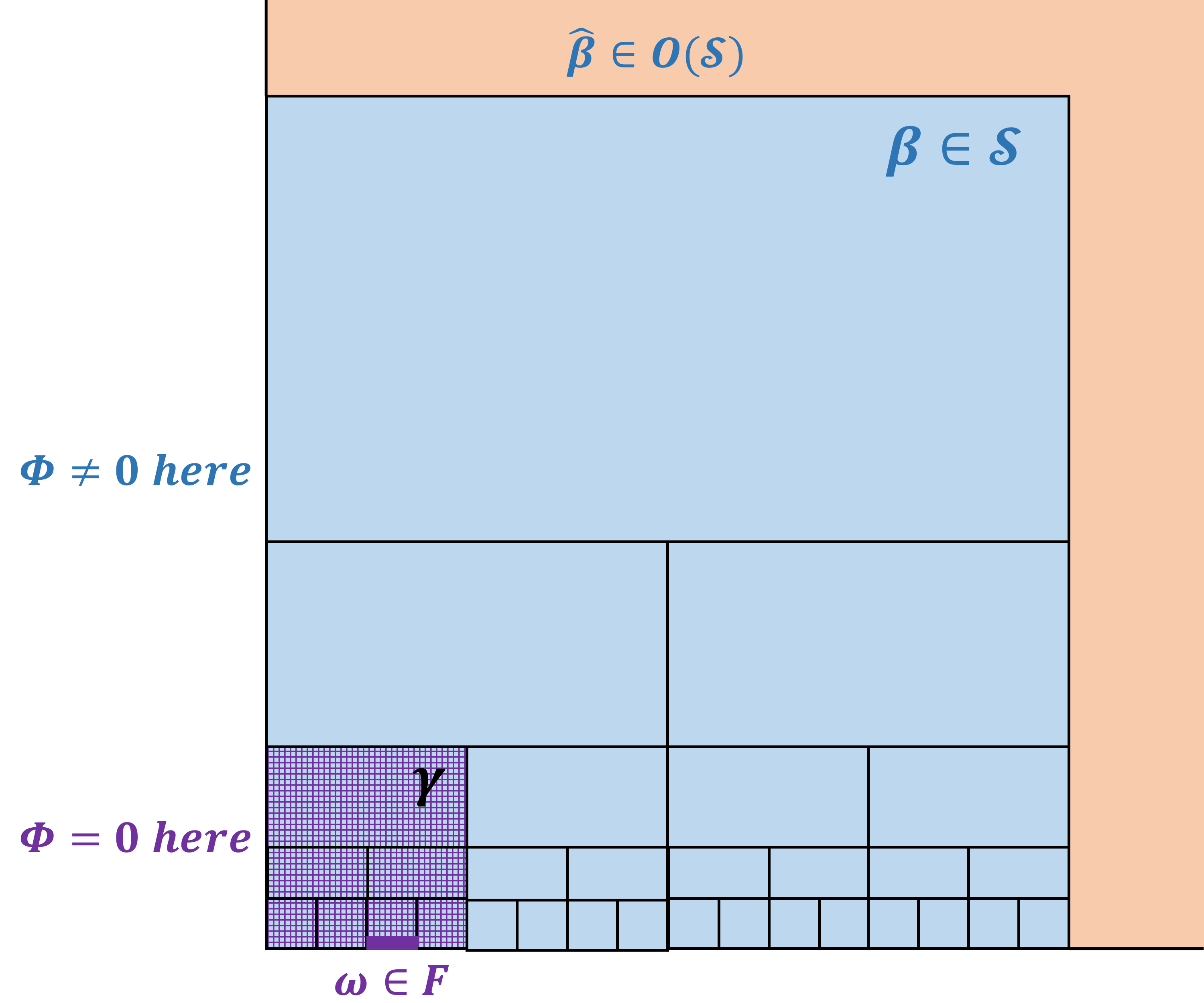}
    \caption{ }
  \end{subfigure}
  \caption{Lemma \ref{Phi}.}
  \label{fig:Pic3}
\end{figure}

So, keeping in mind \eqref{E:If}, we have:
	\begin{eqnarray*}
	I\Phi(\omega) &=& \sum_{\beta\geq\alpha>\gamma} \frac{1}{\lambda} If(\beta) I^*\sigma(\alpha)\\
		&=& \frac{1}{\lambda} If(\omega) \bigg( 
		\underbrace{\sum_{\beta\geq\alpha\geq\gamma} I^*\sigma(\alpha)}_{>\lambda \text{ because } \Phi(\gamma)=0} - 
		\underbrace{I^*\sigma(\gamma)}_{<1 \text{ by } \eqref{mass1}} \bigg)\\
	&>& \frac{\lambda-1}{\lambda} If(\omega).
	\end{eqnarray*}
If we assume, for example, that $\lambda > 3$, then $\frac{\lambda-1}{\lambda}>\frac{2}{3} =: c$.

\textbf{Case 2:} $\Phi(\omega) \neq 0$. Let $\widehat{\beta} \in O(\cS)$ be the dyadic parent of $\beta\in\cS$. Then
	\begin{eqnarray*}
	I\Phi(\omega) &=& \sum_{\beta\geq\alpha\geq\omega} \frac{1}{\lambda} If(\beta) I^*\sigma(\alpha)\\
		&=& \frac{1}{\lambda} If(\omega) \bigg( 
			\underbrace{V^{\sigma}(\omega)}_{\geq\lambda \text{ by \eqref{gela}}} - 
			\underbrace{V^{\sigma}(\widehat{\beta})}_{<1 \text{ by \eqref{le1}}} \bigg)\\
		&>& \frac{\lambda-1}{\lambda} If(\omega).
	\end{eqnarray*}


To prove the energy estimate \eqref{enest1}, let us recall that 
\begin{align*}
&\|\Phi\|_{\ell^2(T)}^2 = \frac1{\la^2} \sum_{\beta \in \cS} [I[f](\beta)]^2\sum_{\al\in Q_\beta: \Phi(\al)\neq 0} I^*_{\cS}[\sigma] (\al)^2=
\\
&  \frac1{\la^2} \sum_{\beta \in \cS} [I[f](\beta)]^2 \int_\beta V^\sigma_{\cS,c}(\om) d \sigma(\om)\,,
\end{align*}
where for $\om \le \beta$, $\beta\in \cS$,
$$
V^\sigma_{\cS, c} (\gamma) := \sum_{\beta\ge \gamma' \ge \gamma: \Phi(\gamma')\neq 0} I^*[\sigma](\gamma')\,.
$$
But $V^\sigma_{\cS, c} (\om)  \le \la$, because this is how $\Phi$ is defined in \eqref{PhiDef}.

Let us introduce a new measure on $T$, called $\sigma_{\cS}^f$, which has masses only on 
vertices $\beta\in \cS$, and each mass is 
$$
\sigma_{\cS}^f(\beta) := I[f](\beta) I^*[\sigma](\beta)\,.
$$
Hence, obviously, we can rewrite the previous estimate of $\|\Phi\|_{\ell^2(T)}^2$ as follows:
\begin{align}
\label{Phi}
&\|\Phi\|_{\ell^2(T)}^2 \le \frac1{\la}\sum_{\beta \in \cS} [I[f](\beta)]^2 I^*[\sigma](\beta)=\notag
\\
& \frac1{\la} \int_{\cS} I[f](\beta) d\sigma_{\cS}^f(\beta) =  \frac1{\la} \sum_{\al \in O(\cS)} f(\al) I^*[\sigma_{\cS}^f](\al) =: \frac1{\la}\mathcal{I}\,,
\end{align}
where $\mathcal{I}:=\int_{\cS} I[f](\beta) d\sigma_{\cS}^f(\beta)$.
To continue, let us make a self estimate of the term $\mathcal{I}$.

\begin{align*}
&\mathcal{I}=\sum_{\al \in O(\cS)} f(\al) I^*[\sigma_{\cS}^f](\al)  \le \|f\|_{\ell^2(T)}\big(\sum_{\al\in O(\cS)} \big(I^*[\sigma_{\cS}^f](\al) \big)^2\Big)^{1/2}=
\\
& \|f\|_{\ell^2(T)} \Big(\int V^{\sigma_S^f} d\sigma_{\cS}^f\Big)^{1/2} = \|f\|_{\ell^2(T)} \Big(\sum_{\beta\in \cS} V^{\sigma_{\cS}^f} (\beta)\sigma_{\cS}^f(\beta)\Big)^{1/2}
\end{align*}

We want to show that $\mathcal{I} \le 8\|f\|_{\ell^2(T)}^2$.
Split $\cS=\cup_{k\in \bZ} \cS_k$, where
$$
\cS_k=\{\beta\in \cS:  2^{k} \le I[f](\beta) < 2^{k+1}\}\,.
$$
$$
\sigma_{\cS_k}^f  =\sigma_{\cS}^f | \cS_k\,.
$$
 We will  estimate now $\sum_{\beta\in \cS} V^{\sigma_{\cS}^f} (\beta)\sigma_{\cS}^f(\beta)\le 64\|f\|_{\ell^2(T)}^2$ as follows:
\begin{align*}
&\sum_{\beta\in \cS} V^{\sigma_{\cS}^f} (\beta)\sigma_{\cS}^f(\beta) \le
\\
&\le 2\sum_k \sum_{\beta\in \cS_k} \sum_{j, k: \,j\le k} V^{\sigma_{\cS_j}^f} (\beta)\sigma_{\cS_k}^f(\beta) \le
\\
&\le 2\sum_k\sum_{\beta\in \cS_k} \sum_{j, k: \,j\le k}  2^{j+1}V^{\sigma}(\beta) \sigma_{\cS_k}^f(\beta) \le  2\sum_k  2^{k+1}V^{\sigma}(\beta) \sigma_{\cS_k}^f(\beta)  \le 
\\
& \le 8\sum_k\sum_{\beta\in \cS_k} I[f](\beta) \sigma_{\cS}^f(\beta) 
 \le 8\sum_{\beta\in \cS}I[f](\beta) \sigma_{\cS}^f(\beta) = 8 \mathcal{I}\le
 \\
& \le 
 8\|f\|_{\ell^2(T)} \big(\sum_{\beta\in \cS} V^{\sigma_{\cS}^f} (\beta)\sigma_{\cS}^f(\beta)\big)^{1/2}\,.
\end{align*}
We used here that by by \eqref{le1} and \eqref{mass1}  $\beta$'s in $\cS$ are all such that
$$
V^{\sigma}(\beta)  \le 2\,.
$$
Therefore,
$$
\mathcal{I} \le 4 \|f\|_{\ell^2(T)}^2\,.
$$
Combining with \eqref{Phi}, we see that the energy estimate \eqref{enest1} is proved, and, thus, the lemma is completely proved.
\end{proof}

\section{Majorization on bi-tree}
\label{T2lemma}

We finish here the proof of  Theorem \ref{2D}.
Let us recall this theorem, it is the following result.

\begin{theorem}\label{2D1}
Let $\mu$ be a positive measure on $\partial T^2$ such that  $\bV \leq 1$ on supp$(\mu)$ and, for some large $\lambda$, $\bV^{\mu}\geq\lambda$ on a set 
$F \subset \partial T^2$. Then there exists a positive function $\varphi$ on $T^2$ such that:
	\begin{itemize}
	\item $\varphi$ satisfies $\bI\varphi(\omega)\geq\lambda$ for all $\omega\in F$.
	\item $\|\varphi\|^2_{\ell^2(T^2)} \leq \frac{C}{\lambda}\cE[\mu]$.
	\end{itemize}
\end{theorem}

\begin{proof}
All of our dyadic rectangles are inside the unit square $Q_0= I_0\times I_0$.

Let us consider the family of dyadic rectangles $\gamma_x\times \alpha_y$ with a fixed vertical side $\alpha_y$, and define
	$$G(\gamma_x) := G^{\alpha_y}(\gamma_x) := \sum_{\alpha'\geq\alpha_y}\mu(\gamma_x\times\alpha').$$
Then note that 
	$$I_xG^{\alpha_y}(\gamma_x) = \bV^{\mu}(\gamma_x\times\alpha_y).$$
Moreover,
	\begin{equation}\label{f1}
	G^{\alpha_y}(\gamma_x) \leq 1,\:\: \forall \gamma_x, \alpha_y.
	\end{equation}
Indeed, let $\tau_y$ be the biggest (if it exists) dyadic $I_0 \geq \tau_y\geq \alpha_y$ such that 
$(\gamma_x \times \tau_y) \cap \text{supp}(\mu) = \emptyset$ (see Figure \ref{fig:Pic4}). Then
	$$
	G^{\alpha_y}(\gamma_x) = \sum_{\substack{\alpha'\geq\alpha_y \\ \alpha' \le\tau_y}} \mu(\gamma_x\times\alpha')
	+ \sum_{\alpha' > \tau_y}\mu(\gamma_x\times\alpha').
	$$
The first term above is obviously $0$, and the second term is $\leq 1$ because it is less than $\bV^{\mu}$ for some point in supp$(\mu)$.
In case $\tau_y$ does not exist, obviously $G^{\alpha_y}(\gamma_x)=0$.

\begin{figure}[h!]
  \centering
    \includegraphics[width=4in]{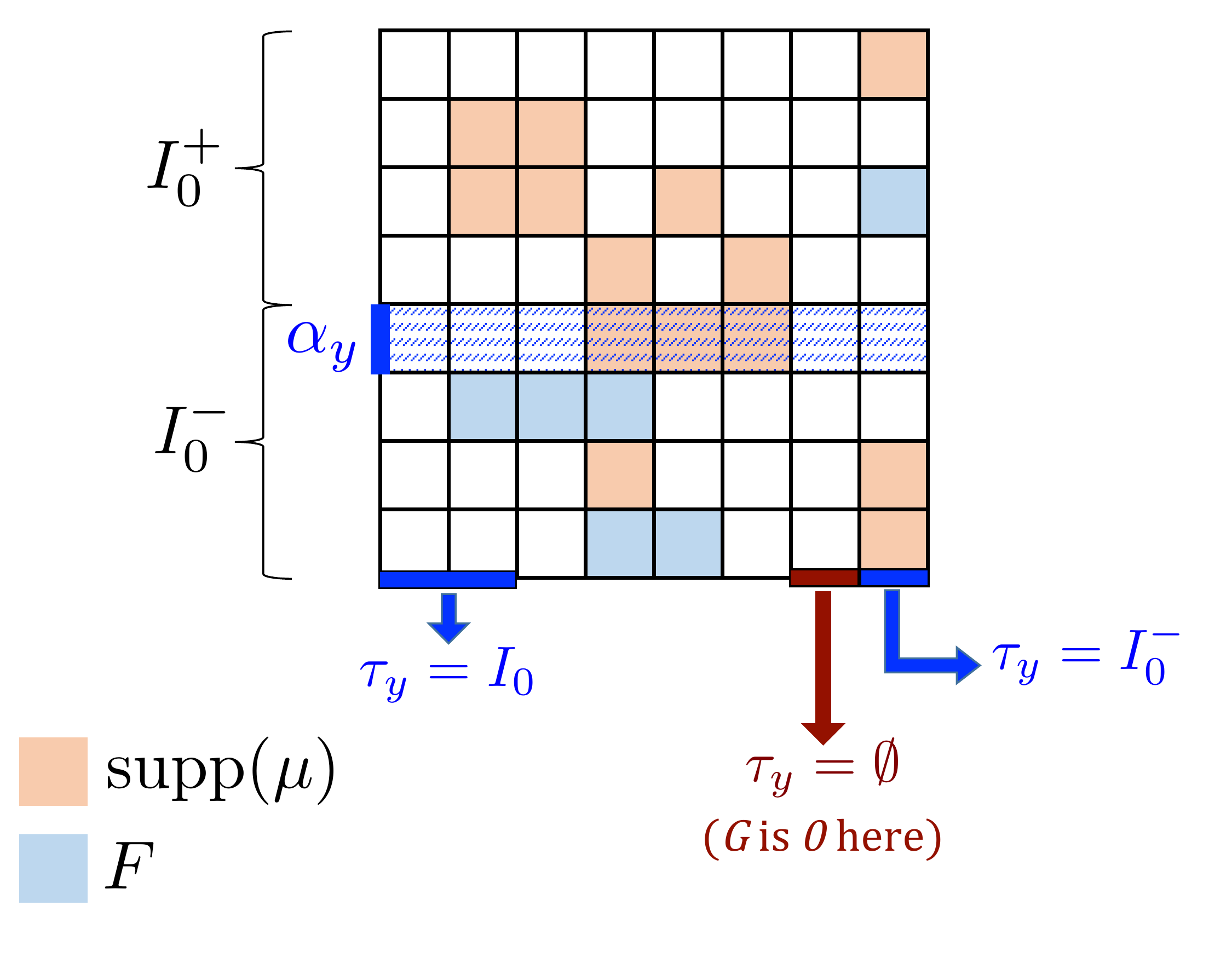}
      \caption{Examples of various $\tau_y$ for a fixed $\alpha_y$.}
  \label{fig:Pic4}
\end{figure}

Now, \eqref{f1} implies that we may consider the family $\cS := \cS(\alpha_y)$ of maximal stopping intervals $\beta_x\in T_x$ such that
$I_xG^{\alpha_y}(\beta_x) = \bV(\beta_x\times\alpha_y)> 1$. Then
	$$I_xG^{\alpha_y}(\beta_x) = \bV(\beta_x\times\alpha_y) \leq 2, \:\:\forall \beta_x\in\cS(\alpha_y).$$
To see this, let $\beta_x\in\cS(\alpha_y)$ and $\widehat{\beta}_x$ be its dyadic parent. Then
$I_xG^{\alpha_y}(\widehat{\beta}_x)\leq 1$, so
	$$I_xG^{\alpha_y}(\beta_x) = \sum_{\gamma_x\geq\beta_x} G^{\alpha_y}(\gamma_x) = 
		\underbrace{G^{\alpha_y}(\beta_x)}_{\leq 1 \text{ by \eqref{f1}}} + \underbrace{I_x G^{\alpha_y}(\widehat{\beta}_x)}_{\leq 1} \leq 2.$$

Another immediate property of the collection $\cS(\alpha_y)$ is
	$$\beta_x\in\cS(\alpha_y) \Rightarrow (\beta_x\times\alpha_y) \cap \text{supp}(\mu) = \emptyset.$$
Otherwise, suppose $Q\in\partial T^2$ is in this intersection. Then 
	$$I_xG^{\alpha_y}(\beta_x) = \bV^{\mu}(\beta_x\times\alpha_y) \leq \bV^{\mu}( Q)\leq 1,$$
a contradiction. It is then obvious that
	\begin{equation}\label{why0}
	\beta_x\in\cS(\alpha_y) \Rightarrow \mu(\beta_x'\times\alpha_y)=0, \:\forall \beta_x'\leq \beta_x.
	\end{equation}

We claim next that
	\begin{equation}\label{f4f5}
	\text{If for some } \omega_x: \bV^{\mu}(\omega_x\times\alpha_y)\geq\frac{\lambda}{3} \text{, then }
		\cS(\alpha_y) \neq\emptyset \text{ and } O(\cS(\alpha_y))\neq\emptyset.
	\end{equation}
Recall that $\lambda$ is large, so obviously $\bV^{\mu}(\omega_x\times\alpha_y) >1$, and then $\cS(\alpha_y)$ is non-empty.
Also, $I_xG^{\alpha_y}(\text{root}_{T_x}) = G^{\alpha_y}(\text{root}_{Tx})\leq 1$, therefore 
any interval in $\cS(\alpha_y)$ is strictly smaller than $I_0$. We therefore have a non-empty family $\cS(\alpha_y)$ of 
largest dyadic intervals in $T_x$ such that $I_xG^{\alpha_y}(\beta_x)>1$, and all these intervals are strictly smaller than $I_0$.

For any small square $\omega=\omega_x\times\omega_y \in F$, let
$\alpha(\omega)$ denote the first from the top (largest) dyadic interval containing $\omega_y$ such that 
	$$\bV^{\mu}(\omega_x\times\alpha(\omega)) \geq \frac{\lambda}{3}.$$
Then by definition
	\begin{equation}\label{v1}
	\bV^{\mu}(\omega_x \times \alpha) \geq \frac{\lambda}{3}, \:\:\forall \alpha: \: \omega_y\leq\alpha\leq\alpha(\omega).
	\end{equation}
In particular, for any $\omega\in F$ and for any $\alpha_y$ such that $\omega_y\leq\alpha_y\leq\alpha(\omega)$, we obtained a family $\cS(\alpha_y)$ of disjoint dyadic subintervals of $T_x$ such that 
	\begin{equation}\label{v12}
	\forall \alpha_y: \omega_y \leq \alpha_y \leq \alpha(\omega) \Rightarrow \cS(\alpha_y)\neq\emptyset, \: O(\cS(\alpha_y))\neq\emptyset.
	\end{equation}	
	
Given $\alpha_y$, we constructed a function $G^{\alpha_y}$ on $T_x \times\alpha_y$, and a family $\cS(\alpha_y)\subset T_x$ of disjoint subintervals. Now we need another function on $T_x\times\alpha_y$, namely
	$$f(\gamma_x) := f^{\alpha_y}(\gamma_x) := \mu(\gamma_x\times\alpha_y).$$
Recall that $W(\cS) = \cup_{\beta\in\cS}Q_{\beta}$.

Fix $\alpha_y$ and construct a special function $\Phi^{\alpha_y}$ as follows.

\begin{itemize}
\item If the dyadic strip $I_0\times\alpha_y$ does not contain any $\omega\in F$, then put $\Phi^{\alpha_y} = 0$.
\item Otherwise (see Figure \ref{fig:Pic5}, \ref{fig:Pic6}), let
	$$
	F_{\alpha_y}:= \{\omega_x: \omega=\omega_x\times\omega_y \in F \text{ s.t. } \omega \text{ lies in } I_0\times\alpha_y \text{ and }
	\alpha_y \leq \alpha(\omega)\}.
	$$
If $F_{\alpha_y}=\emptyset$, again put $\Phi^{\alpha_y} = 0$. Otherwise, for some $\omega_x\in F_{\alpha_y}$, by \eqref{v12}:
	$$\alpha_y \leq \alpha(\omega) \Rightarrow \cS(\alpha_y)\neq\emptyset \text{ and } O(\cS(\alpha_y))\neq\emptyset.$$
We claim that we are now in the situation of Lemma \ref{Phi}.
\end{itemize}

\begin{figure}[h!]
  \centering
    \includegraphics[width=3in]{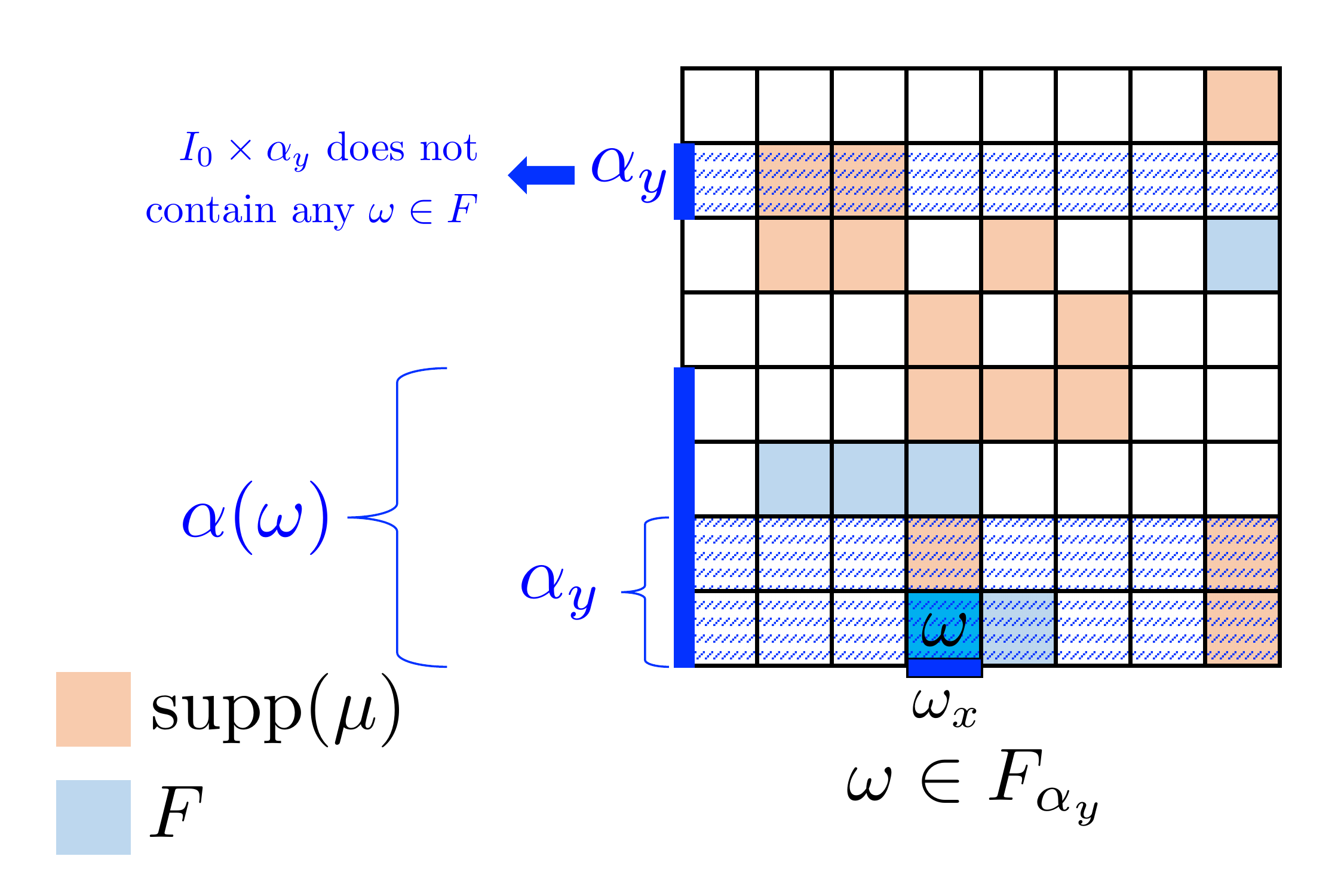}
      \caption{Construction of the function $\Phi^{\alpha_y}$ (1).}
  \label{fig:Pic5}
\end{figure}

\begin{figure}[h!]
  \centering
    \includegraphics[width=3in]{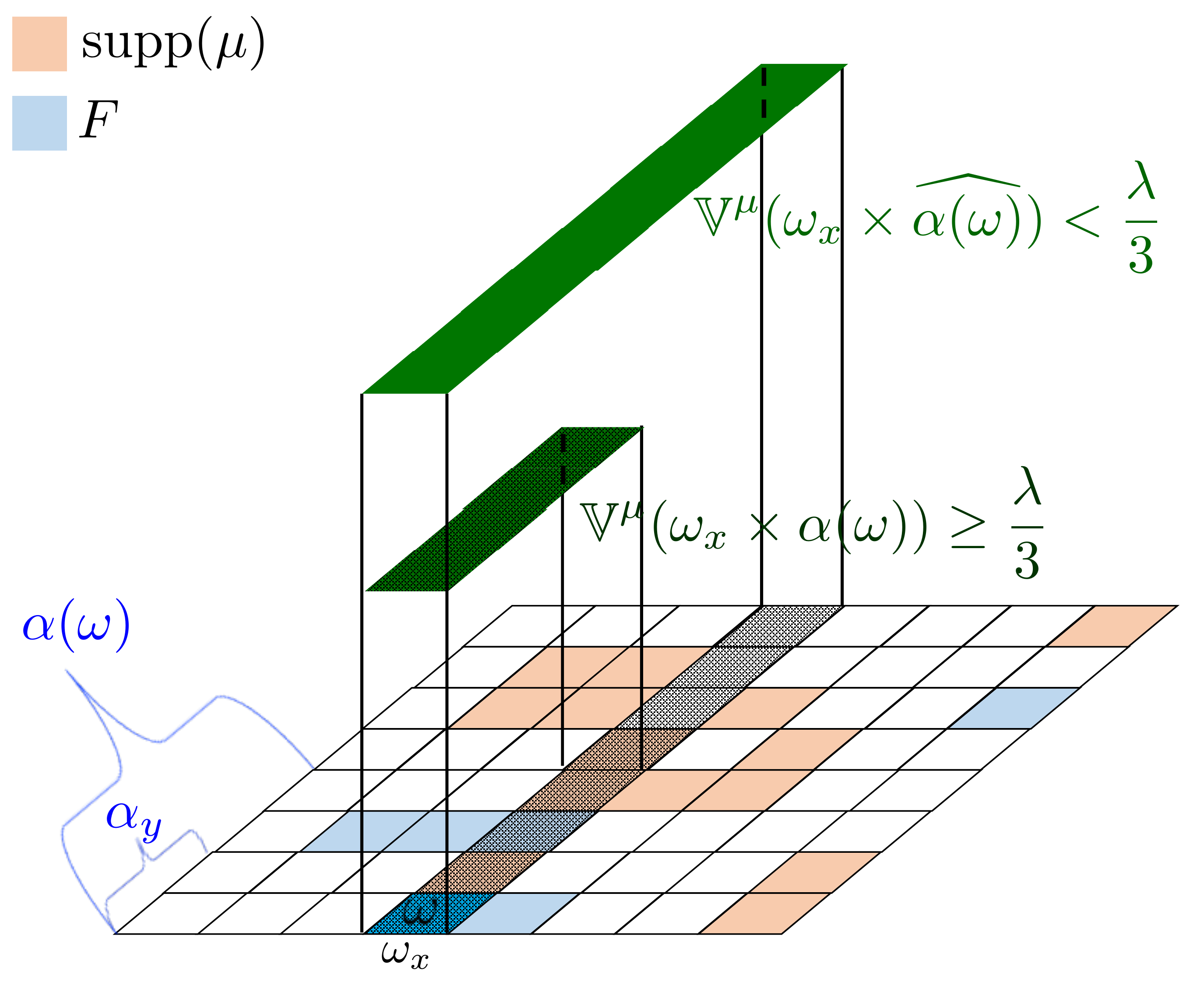}
      \caption{Construction of the function $\Phi^{\alpha_y}$ (2).}
  \label{fig:Pic6}
\end{figure}

Let $\sigma$ be a measure on $\partial T_x$ defined by:
	$$\sigma(\omega_x) := \sum_{\alpha'\geq\alpha_y} \mu(\omega_x\times\alpha'), \:\forall \omega_x\in\partial T_x.$$
Then
	$$G^{\alpha_y}(\gamma_x) = I_x^*\sigma(\gamma_x).$$

By \eqref{v1}:
	$$\frac{\lambda}{3} \leq \bV^{\mu}(\omega_x\times \alpha_y) = IG^{\alpha_y}(\omega_x) = II^*\sigma(\omega_x) = V^{\sigma}(\omega_x),$$
so $F_{\alpha_y} \subset \partial T_x \cap W(\cS(\alpha_y))$ -- otherwise, we would have $\bV^{\mu}(\omega_x\times\alpha_y)\leq 1$, a contradiction. We make note of the fact that
	\begin{equation}\label{B}
	V^{\sigma}(\omega_x) \geq \frac{\lambda}{3}, \:\forall \omega_x\in F_{\alpha_y}.
	\end{equation}
Also, by definition of $\cS(\alpha_y)$,
	$$V^{\sigma}(\gamma_x) = \bV^{\mu}(\gamma_x\times\alpha_y) < 1, \:\forall \gamma_x\in O(\cS(\alpha_y)).$$
	
By \eqref{why0},
	\begin{equation}\label{A}
	f^{\alpha_y} = 0 \text{ on } W(\cS(\alpha_y)).
	\end{equation}
	
So, we are now indeed under the assumptions of Lemma \ref{Phi}, so we have a non-negative function $\Phi^{\alpha_y}$ on $T_x$ such that, with positive absolute constants $c$, $C$:
	
	\begin{equation}\label{vf1}
	I\Phi^{\alpha_y}(\omega_x) \geq c If^{\alpha_y}(\omega_x), \:\forall \omega_x \in F_{\alpha_y}.
	\end{equation}
	
	\begin{equation}\label{vf2}
	\|\Phi^{\alpha_y}\|^2_{\ell^2(T_x)} \leq \frac{C}{\lambda} \|f^{\alpha_y}\|^2_{\ell^2(T_x)}.
	\end{equation}

Now put
	$$\varphi(\gamma_x, \alpha_y) := \Phi^{\alpha_y}(\gamma_x).$$
Summing \eqref{vf2} over all $\alpha_y \in T_y$:
	\begin{eqnarray*}
	\|\varphi\|^2_{\ell^2(T^2)} &=& \sum_{\gamma_x, \alpha_y} \bigg(\Phi^{\alpha_y}(\gamma_x)\bigg)^2 = \sum_{\alpha_y} \|\Phi^{\alpha_y}\|^2_{\ell^2(T_x)}\\
		&\leq& \frac{C}{\lambda} \sum_{\gamma_x, \alpha_y} \mu(\gamma_x\times\alpha_y)^2 = \mathcal{E}[\mu].
	\end{eqnarray*}
Given $\omega=\omega_x\times\omega_y \in F$, sum \eqref{vf1} in $\alpha_y$: $\omega_y\leq\alpha_y\leq\alpha(\omega)$:
	\begin{eqnarray*}
	\bI\varphi(\omega) &=& \sum_{\alpha_y\geq\omega_y} I\Phi^{\alpha_y}(\omega_x) 	
		\geq \sum_{\alpha_y: \omega_y\leq\alpha_y\leq\alpha(\omega)} I\Phi^{\alpha_y} (\omega_x)
		\\
		&\geq& c \sum_{\alpha_y: \omega_y\leq\alpha_y\leq\alpha(\omega)} If^{\alpha_y} (\omega_x) = 
			c \sum_{\alpha_y: \omega_y\leq\alpha_y\leq\alpha(\omega)} \sum_{\omega'\geq\omega_x} \mu(\omega'\times\alpha_y)
			\\
		&=& c\bigg( \sum_{\substack{\omega'\geq\omega_x
		 \\
		  \alpha'\geq\omega_y}} \mu(\omega'\times\alpha') -
			\sum_{\substack{\omega'\geq \omega_x 
			\\
			 \alpha'>\alpha(\omega)}} \mu(\omega'\times\alpha') \bigg)\\
		&=& c\bigg( \underbrace{\bV^{\mu}(\omega_x\times\omega_y)}_{\geq\lambda \text{ because } \bV^{\mu}\geq\lambda} - 
		\underbrace{\bV^{\mu}(\omega_x\times\widehat{\alpha(\omega)})}_{<\lambda/3 \text{by defn. of } \alpha(\omega)} \bigg)\\
		&\geq& c\frac{2\lambda}{3}.
	\end{eqnarray*}
\end{proof}

\section{The proof of Lemma \ref{l:l2}}
\label{pr32}

The proof of Lemma \ref{l:l2} is also based on Theorem \ref{2D1}, but rather on a  modification of it.
Hence we need a a special modification of Theorem \ref{2D1}.  Let
$$
E_1:= \{ (\tau\times \alpha): \bV^{\mu} (\tau\times \alpha)<1\}\,.
$$
This set can be empty because we do not assume anything on $\mu\ge 0$ at this moment.
Put 
$$
{\bV}_1^{\mu} (\tau\times \alpha):= \sum_{\tau'\ge \tau, \alpha'\ge \alpha, (\tau', \alpha')\in E_1} \mu(\tau'\times\alpha')\,.
$$
For any positive function on $T^2$ we denote
$$
{\bI}_1 \vf := \sum_{\tau'\ge \tau, \alpha'\ge \alpha, (\tau', \alpha')\in E_1}  \vf(\tau'\times \alpha')\,.
$$
Denote $\cE_1[\mu]:=\int \bV_1^{\mu} \, d\mu$.
Then, 
$$
\cE_1[\mu]=\int \bV_1^{\mu} \, d\mu = \sum_{\tau\times \alpha \in E_1} \big(\mu (\tau\times \alpha)\big)^2\,.
$$

\begin{theorem}
\label{2D2}
Let $\mu$ is a positive measure on $\pd T^2$ such that  $\bV_1^\mu\ge \la>> 1$ on a set $F\subset \pd T^2$. 
Then there exists positive  $\vf$ on $T^2$ such that
\begin{itemize}
\item $\vf$ satisfies $\bI\vf (\om) \ge \lambda$ for all $\om \in F$,
\item $\|\vf\|_{\ell^2(T^2)}^2 \le \frac{C}{\lambda} \cE_1[\mu]$.
\end{itemize}
\end{theorem}

\begin{proof} 
If $E_1=\emptyset$, there is nothing to prove as the set $F$ of large values of $\bV_1^\mu$ will be empty (since $\bV_1^\mu=0$ identically).

Now we follow closely the proof of Theorem \ref{2D1}. Again fix $\alpha_y \in T_y$. As before we introduce two function (notice the modification):
$$
g_1(\tau_x) := \sum_{\alpha_y' \ge \alpha_y, (\tau_x\times \alpha_y')\in E_1} \mu(\tau_x\times \alpha_y')\,,\
$$
$$
f_1(\tau_x) := \mu (\tau_x\times \alpha_y), \quad \tau_x\times \alpha_y\in E_1; \quad 0\,\,\text{otherwise}\,.
$$
Of course we should keep in mind that these functions have implicit superscript $\alpha_y$.
Notice that
$$
I g_1(\gamma_x) = \sum_{\gamma_x'\ge \gamma_x, \alpha_y'\ge \alpha_y, (\gamma_x'\times \alpha_y')\in E_1} \mu(\gamma_x'\times \alpha_y') = \bV_1^\mu (\gamma_x\times \alpha_y)\,.
$$
So,  consider the family $\cS=\cS^{\alpha_y}$ of maximal dyadic intervals (=nodes of $T_x$) such that
\begin{equation}
\label{ge1}
Ig_1(\beta_x) \ge 1\,.
\end{equation}
As before consider $W(\cS)$ and $O(\cS)$.
Given $E_1\neq \emptyset$, we conclude that for some $\alpha_y$ the set $O(\cS)$ is non-empty and that
\begin{equation}
\label{small2}
I g_1 < 1\quad \text{on}\,\, O(\cS)\,.
\end{equation}
Consider
$$
	F_{\alpha_y}:= \{\omega_x: \omega=\omega_x\times\omega_y \in F \text{ s.t. } \omega \text{ lies in } I_0\times\alpha_y \text{ and }
	\alpha_y \leq \alpha(\omega)\}.
	$$
	Now $\al(\om)$ is computed with respect to potential $\bV_1^\mu$: the largest such $\al$ that $\bV_1(\om_x\times \al) \ge \frac{\la}{3}$.

Non-emptiness of $E_1$ also implies $\mu (I_0\times I_0) < 1$ and thus \eqref{small2} can be complemented by
\begin{equation}
\label{beta2}
I g_1 \le 2\quad \text{for all }\,\, \beta \in \cS\,.
\end{equation}
However, if $F\cap\big(T_x\times \alpha_y\big)\neq \emptyset$, then  on $F_{\alpha_y}\subset \pd T_x$
\begin{equation}
\label{gela3}
Ig_1 \ge \frac{\la}{3}\,.
\end{equation}

Next,  following the scheme of the proof of Theorem \ref{2D1}, let us check  that
\begin{equation}
\label{fWS0}
f_1 =0 \quad\text{on} \,\, W(\cS)\,.
\end{equation}
Indeed, let $\gamma \in W(\cS)$, so there exists $\beta\in \cS$ such that $\gamma \le \beta$. Then, using \eqref{ge1}, we get
$$
\bV_1^\mu(\gamma\times \alpha_y) = I g_1(\gamma) \ge Ig_1(\beta) \ge 1,
$$
and, hence, by the definition of $f_1$, $f_1(\gamma)=0$.

\bigskip

We are almost in the assumptions of Lemma \ref{Phi}. In fact, we have $W(\cS)$, $O(\cS)$, function $f_1$ that plays the part of $f$ and function $g_1$ that plays the part of $G$, and we have assumption \eqref{small2} that is like \eqref{le1} and assumption \eqref{gela3} that is like assumption \eqref{gela}.
There is a difference though, because the property $G=I^*[\sigma]$ is missing, $g_1$ is more complicated. But we will be able to circumvent this difficulty in a rather easy way.

\medskip

It is clear that we are interested only in those $\alpha_y$, for which $f_1\neq 0$, therefore, we are interested only in those $\alpha_y$, for which $O(\cS)\neq \emptyset$.

Remembering this, next consider \eqref{gela3}. If \eqref{gela3} happens (there are many $\alpha_y$'s for which this will happen, namely, those for which $F\cap \big( T_x\times \alpha_y)\ne \emptyset$), then, obviously, \eqref{gela3} may happen only on the part of $\pd T_x$ that lie inside some of the intervals $\beta\in \cS$.

To reduce everything to Lemma \ref{Phi} we will need one property of $g_1$ that will replace the property $G=I^*[\sigma]$ that  is missing. 
 Namely, we have 
 \begin{lemma}
 \label{superadd}
 Let $\tau_x=\tau_x^1\cup \tau_x^2$, $\tau_x^i$ being two children of $\tau_x$. Then
 $$
 g_1(\tau_x) \ge g_1(\tau_x^1) + g_1(\tau_x^2)\,.
 $$
 \end{lemma}
 \begin{proof}
 Let $\alpha_y^i\ge \alpha_y$ be the smallest interval such that $\tau_x^i\times \alpha_y^i$ belongs to $E_1$.  
 And let $\tau_x\times \hat \alpha_y$  be the smallest interval such that $\tau_x\times \hat \alpha_y$ belongs to $E_1$. 
 Without the loss of generality we assume that $\alpha_y^1\le \alpha_y^2$. 
 Then (see Figure \ref{fig:Pic7}) $\tau_x\times \alpha_y^1$ contains $\tau_x^1\times \alpha_y^1\in E_1$, and we conclude that
 $\tau_x \times \alpha_y^1$ also belongs to $E_1$. But $\tau_x\times  \hat\alpha_y$ is the smallest such  rectangle. Therefore,
 $$
 \tau_x\times  \hat\alpha_y \subset \tau_x\times \alpha_y^1, \quad \text{and so}\quad \hat\alpha_y \le \alpha_y^1\le \alpha_y^2\,.
 $$
 
 \begin{figure}[h!]
  \centering
    \includegraphics[width=2in]{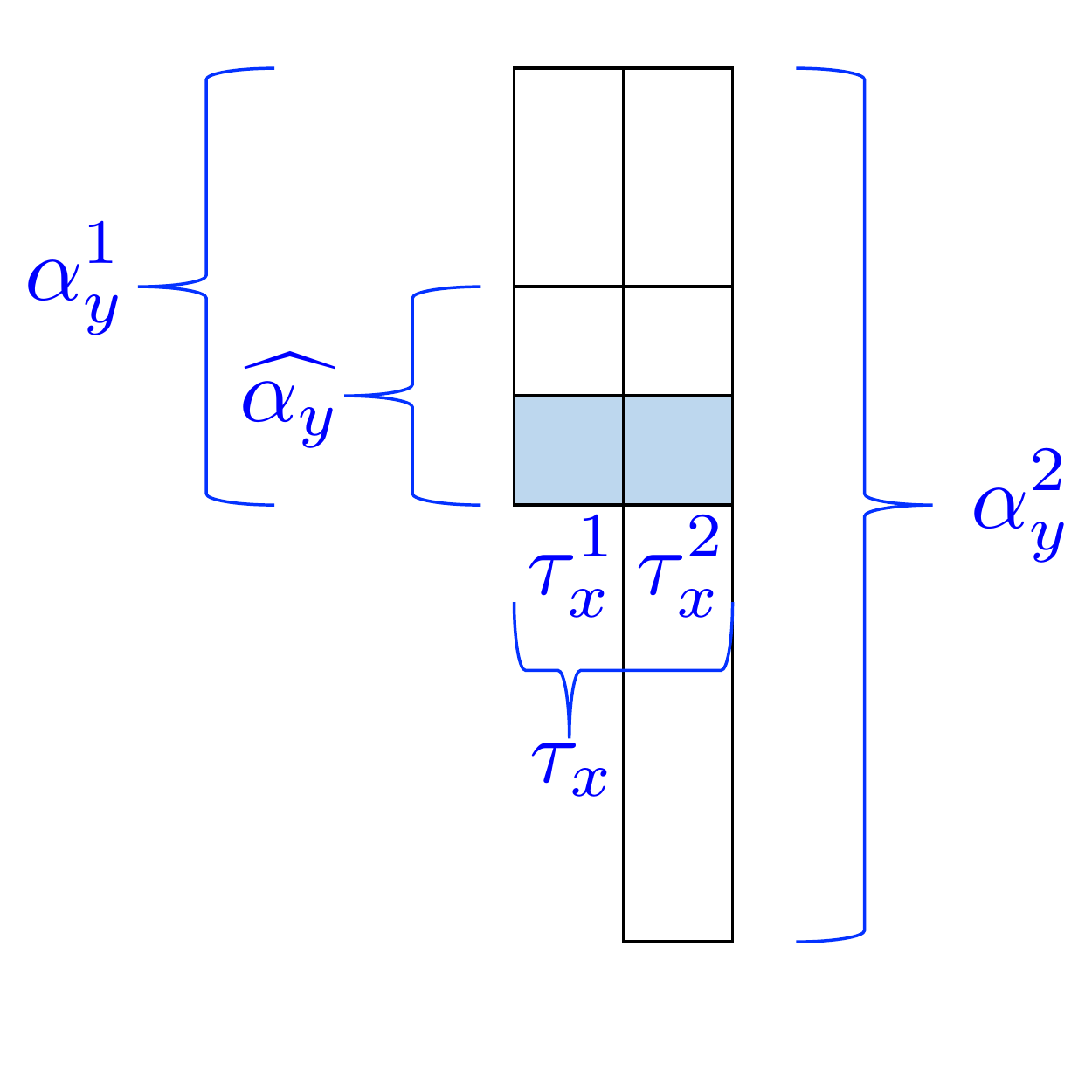}
      \caption{Lemma \ref{superadd} construction.}
  \label{fig:Pic7}
\end{figure}

In the definition of $g_1(\tau_x)$ we have the sum of $\mu$'s over $\tau_x \times \alpha, \alpha= Gen_k \hat\alpha_y$, $k\ge 0$, where $Gen_k(I)$ means the predecessor  of $I$, which is $2^k$ times larger than $I$. We write
\begin{align*}
&g_1(\tau_x) =\sum_k \mu(\tau_x \times Gen_k\hat\alpha_y) = \sum_{k=0}^\infty \mu(\tau_x^1 \times Gen_k\hat \alpha_y)  + \sum_{k = 0}^\infty\mu(\tau_x^2 \times Gen_k\hat \alpha_y) \ge
\\
&\sum_{k=0}^\infty \mu(\tau_x^1 \times Gen_k \alpha_y^1)  + \sum_{k=0}^\infty \mu(\tau_x^2 \times Gen_k \alpha_y^1) = g_1(\tau_x^1) + g_1(\tau_x^2)\,,
\end{align*}
where the inequality holds because there are less predecessors for larger intervals.

 \end{proof}
 
 \begin{defin}
 Function $g$ satisfying $ g(\tau) \ge g(\tau^1) + g(\tau^2)$ for any $\tau\in T$ and its two children $\tau^1, \tau^2$ is called two point super-harmonic. Function $G$ satisfying $ G(\tau) = G(\tau^1) + G(\tau^2)$ for any $\tau\in T$ and its two children $\tau^1, \tau^2$ is called two point harmonic. 
 \end{defin}
 
 This property of $g_1$ implies immediately the following property of $Ig_1$:
 \begin{lemma}
 \label{superharm}
 Function $Ig_1$ on $T$ is three point super-harmonic.  In other words,
 let $\tau\in T$  has two children $\tau^1, \tau^2$ and father $\tau_3$. Then 
 $$
 Ig_1(\tau) \ge \frac13\big(Ig_1(\tau_1) + Ig_1(\tau_2) +I g_1(\tau_3)\big)\,.
 $$
 \end{lemma}
 \begin{proof}  
 Let $c= g_1(\tau), a= g_1(\tau_1), b= g_1(\tau_2)$. The above mentioned inequality is obviously equivalent to saying that
 $$
 \frac13(a+c)  +\frac13(b+c) \le c\,.
 $$
 This is of course true by Lemma  \ref{superadd}.

 \end{proof}
 
  \begin{remark}
 \label{harm}
 Notice that this claim simultaneously proves that if $\sigma$ is a positive measure on $\pd T$ and if $G(\tau):= I^*\sigma(\tau), \tau\in T$, then
 $IG= V^\sigma$ is three point harmonic. Indeed, 
 if we use the same proof with $IG=V^\sigma$ replacing $Ig_1$, 
 we would come to $c=a+b$, which is $I^*\sigma(\tau)= I^*\sigma(\tau_1) + I^*\sigma(\tau_2)$ which is of course correct.
 \end{remark}

Now  let us use \eqref{gela3} as follows. Let $\rho$ be an equilibrium measure on $F_{\al_y}=Proj_{T_x}\big[F\cap \big(T_x\times\alpha_y\big)\big]$. 
In particular $V^\rho=1$ on $F_{\al_y}$. Denote
$$
\sigma:=\frac{\lambda}{3} \rho\,.
$$
Then by \eqref{gela3} we have:
\begin{equation}
\label{eqla3}
V^\sigma=\frac{\lambda}{3}, \quad \text{on} \quad  F_{\al_y}=Proj_{T_x}\big[F\cap \big(T_x\times\alpha_y\big)\big]\,.
\end{equation}

\medskip

 \begin{remark}
 \label{falseMinP}
 One can now think that maximum principle on tree $T$ would now imply that super-harmonic $Ig_1$
 is bigger than harmonic $IG, G= I^*\sigma$, on the whole tree $T$ because on the boundary they satisfy \eqref{eqla3}. However,  this is not the right reasoning because of two important obstacles: 1) \eqref{eqla3} holds not on the whole boundary of $T$ but only on some part of it; 2) for $3$ point subharmonic functions minimum principle claims that minimum is either on the boundary or at the root of the tree. And we have seemingly no information about the behavior of super-harmonic $Ig_1$ and harmonic $IG= I (I^*\sigma)$ at the root.
 One needs another minimum principle. It is in Lemma \ref{2ptMinP} below.
  \end{remark}

\medskip

Denote $G:= I^*\sigma$. It is a two point harmonic function, and the set of the boundary $\pd T$, where 
it is strictly positive is by definition inside $\supp\sigma=\supp\rho$. So on the set, where $G$ is strictly
positive we have $IG=V^\sigma \le Ig_1$ by \eqref{gela3} and \eqref{eqla3}.

Hence, we are in a position to use Lemma \ref{2ptMinP} and Remark \ref{harm} that imply
$$
 V^\sigma \le Ig_1\quad \text{on}\quad T\,.
 $$
 This and \eqref{small2} gives
 \begin{equation}
 \label{Vsigma}
 V^\sigma< 1\quad \text{on}\quad O(\cS)\,.
 \end{equation}
 Now \eqref{eqla3} and \eqref{Vsigma} correspond to \eqref{le1} and \eqref{gela} of Lemma \ref{Phi}. We use this lemma and get $\Phi$ claimed in it. Then the end of the proof of Theorem \ref{2D2} repeats verbatim the reasoning of Section \ref{T2lemma}.

\end{proof}

\begin{lemma}
\label{2ptMinP}
Let $g, G$ be two non-negative functions on $T$. Let $g$ be two point super-harmonic, and $G$ be two point harmonic functions. Assume that $IG\le Ig$ on the set $P=\{\om \in \pd T: G(\om)>0\}$.
Then $IG \le Ig$ on the whole tree $T$.
\end{lemma}

\begin{proof}
Assume that at a certain $\beta\in T$ we have 
$Ig(\beta)< IG(\beta)$. If simultaneously $g(\beta) < G(\beta)$ we call this $\beta$ good. If it is not good, thus,  $g(\beta) \ge G(\beta)$, then clearly $Ig(\beta_1)< IG(\beta_1)$, where $\beta_1$ denotes the father of $\beta$. Again we query whether $\beta_1$ is good. If not we come to $\beta_2$, which is the father of $\beta_1$. Eventually we will find a good vertex. May be it will be the root of the tree, where $Ig=g, IG=G$.

As soon as we find good $\gamma\in T$, that is $\gamma$ such that simultaneously
\begin{equation}
\label{gG}
Ig(\gamma) < IG(\gamma)
\end{equation}
 and $g(\gamma) < G(\gamma)$, we notice that one of the children
$\gamma_\pm$ (let us call it $\gamma_1$) will also satisfy $g(\gamma_1)< G(\gamma_1)$.
In fact,
$$
g(\gamma_+)+ g(\gamma_-) \le g(\gamma) < G(\gamma)=G(\gamma_+)+ G(\gamma_-) \,.
$$
Now, by recursion, we find a child $\gamma_2$ of $\gamma_1$ such that $g(\gamma_2)< G(\gamma_2)$. We continue doing that till we come to the boundary, namely, to a certain $\gamma_n=:\om\in \pd T$, such that $g(\gamma_n)< G(\gamma_n)$. Vertices $\gamma_1, \dots, \gamma_n$ form the branch of the tree from $\gamma_1$ till $\gamma_n=\om\in \pd T$. We can now add all inequalities $g(\gamma_i)< G(\gamma_i)$, $i=1, \dots, n$, and also add to this inequality \eqref{gG}. 

As a result we get two things: one is that $G(\om) > g(\om)\ge 0$ (that is $\om$ lies in the set $P$),
the second one is
$$
Ig(\om) < IG(\om)\quad \om \in P\,.
$$
But this is a contradiction to the assumption that $Ig\ge IG$ on $P$.

\end{proof}

Define
$$
E_\delta:= \{ (\tau\times \alpha): V^{\mu} (\tau\times \alpha)<\delta\}\,.
$$
Put 
$$
{\bV}_\delta^{\mu} (\tau\times \alpha):= \sum_{\tau'\ge \tau, \alpha'\ge \alpha, (\tau', \alpha')\in E_\delta} {\bI}\mu(\tau'\times\alpha')\,.
$$
For any positive function on $T^2$ we denote
$$
{\bI}_\delta \vf := \sum_{\tau'\ge \tau, \alpha'\ge \alpha, (\tau', \alpha')\in E_\delta}  \vf(\tau'\times \alpha')\,.
$$
Denote $\cE_\delta[\mu]:=\int \bV_\delta^{\mu} \, d\mu$.
Then, 
$$
\cE_\delta[\mu]=\int \bV_\delta^{\mu} \, d\mu = \sum_{\tau\times \alpha \in E_\delta} \big(\mu (\tau\times \alpha)\big)^2\,.
$$

Let $\delta\in (0,1]$. By rescaling $\mu:=\mu/\delta$ we get
\begin{theorem}
\label{2D2r}
Let $\mu$ is a positive measure on $\pd T^2$ such that  $\bV_\delta^\mu\ge  \lambda\ge 1$ on a set $F\subset \pd T^2$. 
Then there exists positive  $\vf$ on $T^2$ such that
\begin{itemize}
\item $\vf$ satisfies $\bI\vf (\om) \ge \lambda$ for all $\om \in F$,
\item $\|\vf\|_{\ell^2(T^2)}^2 \le C\frac{\delta}{\lambda} \cE_\delta[\mu]$.
\end{itemize}
\end{theorem}

\begin{lemma}
\label{cEcE}
Assume that $\mu$ is a positive measure on $\pd T^2$  such that $\bV^\mu\ge 1$ on $\supp\mu$. Then
\begin{equation}
\label{lemma32}
\cE_\delta[\mu] \le C\delta^{1/2} \cE[\mu]\,.
\end{equation}
In particular,
\begin{align*}
&\cE_{T^2\cap \{\bV^\mu \ge \delta\}}[\mu]=\sum_{R \subset \pd T^2\cap \{\bV^\mu \ge \delta\}}\mu(R)^2  \ge \sum_{\alpha\in T^2: \bV^\mu(\alpha)\ge \delta} \big[\bI^*\mu(\alpha)\big]^2= 
\\
& \cE[\mu] -  \cE_\delta [\mu] \ge (1-C\delta^{1/2}) \cE[\mu]\,.
 \end{align*}
\end{lemma}
\begin{proof}
If the first display inequality is proved, then the second  display inequality follows because  given $\alpha\in T^2$ such that $\bV^\mu(\alpha) \ge \delta$, we immediately see that for each point $x\in \supp\mu$ of the dyadic rectangle $R$ corresponding to $\alpha$ we have $\bV^\mu(x) \ge \delta$.

\medskip

To prove the first inequality we will use Theorem \ref{2D2r}. Fix a small positive $\eps$ to be chosen soon.  
Consider $E_k\subset \pd T^2$ such that $E_k=\{x\in \supp \mu: 2^{k-1}< \bV^{\mu}(x) \le 2^k\}$, $k=-\eps\log\frac1{\delta}, \dots ,0, 1,\dots $. Then construct $\Phi_k$ from Theorem \ref{2D2r} with data $\la=2^k$, $\delta$. Then
\begin{align*}
& 2^k \mu(E_k) \le \int_{E_k} \bI( \Phi_k) d\mu \le   \int  \bI(\Phi_k) d\mu   = \sum_{T^2} \Phi_k \bI^*[\mu] \le
\\
& \|\Phi_k\|_{\ell^2} \cE[\mu]^{1/2} \le \Big(\frac{\delta}{2^k}\Big)^{1/2} \cE_\delta [\mu]^{1/2}\cE[\mu]^{1/2}\,.
\end{align*}
Now sum over $k$ and use that $\|\mu\|\le \int \bV^\mu \, d\mu=\cE[\mu]$ as $\bV^\mu\ge 1$ on $\supp\mu$:
\begin{align*}
&\cE_\delta[\mu] =\int \bV^\mu_\delta\, d\mu = \int_{\bV^\mu_\delta\le \delta^\eps} \bV^\mu_\delta\, d\mu + \int_{\bV^\mu_\delta >\delta^\eps} \bV^\mu_\delta\, d\mu  \le \delta^\eps \|\mu\|+ 2 \sum_{k=0}^\infty 2^k \mu(E_k) \le  
\\
&\delta^\eps \cE[\mu]+ 2 \sum_{k=0}^\infty 2^k \mu(E_k)\le  \delta^\eps \cE[\mu]+ C\cE_\delta [\mu]^{1/2}\cE[\mu]^{1/2}  \Big(\frac{\delta}{\delta^\eps}\Big)^{1/2}\,.
\end{align*}
One of the terms on the right is bigger than another. Thus, either
$
\cE_\delta[\mu] \le C\delta^\eps \cE[\mu]
$
or 
$\cE_\delta[\mu]  \le C\delta^{1-\eps} \cE[\mu]$.  Either way, choosing $\eps=\frac12$ we get the result of the lemma.

\end{proof}

The second display inequality of Lemma \ref{cEcE} proves Lemma \ref{l:l2}.

\end{document}